\pdfoutput=1
\RequirePackage{ifpdf}
\ifpdf 
\documentclass[pdftex]{sigma}
\else
\documentclass{sigma}
\fi

\usepackage{mathtools}
\usepackage{stmaryrd,authblk}
\usepackage{mathrsfs}
\usepackage{faktor}
\usepackage{tcolorbox}
\usepackage[linewidth=1pt,
middlelinecolor= black,
middlelinewidth=0.4pt,
roundcorner=1pt,
topline = false,
rightline = false,
bottomline = false,
rightmargin=0pt,
skipabove=0pt,
skipbelow=0pt,
leftmargin=-1cm,
innerleftmargin=1cm,
innerrightmargin=0pt,
innertopmargin=0pt,
innerbottommargin=0pt]{mdframed}

\usepackage{multicol}
\usepackage{array}
\usepackage{shuffle}
\usepackage{mleftright}
\usepackage{zref-base}

\usepackage{tikz}
\usepackage{fp}
\usepackage{ifthen}
\usepackage{calc}
\usetikzlibrary{automata, positioning, arrows}

\numberwithin{equation}{section}

\theoremstyle{definition}
\newtheorem{defi}{Definition}[section]
\newtheorem{Eg}[defi]{Example}
\newtheorem{Egs}[defi]{Examples}
\newtheorem{Rq}[defi]{Remark}
\newtheorem{Not}[defi]{Notation}

\newtheorem{innercustomgeneric}{\customgenericname}
\providecommand{\customgenericname}{}
\newcommand{\newcustomtheorem}[2]{%
	\newenvironment{#1}[1]
	{%
		\renewcommand\customgenericname{#2}%
		\renewcommand\theinnercustomgeneric{##1}%
		\innercustomgeneric
	}
	{\endinnercustomgeneric}
}
\newcustomtheorem{customdefi}{Definition}
\newcustomtheorem{customrq}{Remark}
\newcustomtheorem{customEgs}{Examples}
\newcustomtheorem{customEg}{Example}

\theoremstyle{plain}
\newtheorem{Prop}[defi]{Proposition}
\newtheorem{Lemme}[defi]{Lemma}
\newtheorem{Cor}[defi]{Corollary}
\newtheorem{thm}[defi]{Theorem}
\newtheorem*{thm*}{Theorem}
\newtheorem*{Lemme*}{Lemma}
\newtheorem*{Prop*}{Proposition}

\newtheorem{innercustomgenerictwo}{\customgenericname}
\providecommand{\customgenericname}{}
\newcommand{\newcustomtheoremplain}[2]{%
	\newenvironment{#1}[1]
	{%
		\renewcommand\customgenericname{#2}%
		\renewcommand\theinnercustomgenerictwo{##1}%
		\innercustomgenerictwo
	}
	{\endinnercustomgeneric}
}
\newcustomtheoremplain{customthm}{Theorem}
\newcustomtheoremplain{customlemma}{Lemma}
\newcustomtheoremplain{customprop}{Proposition}

\makeatletter
\zref@newprop{mzheight}[0pt]{\the\ht\z@}
\zref@newprop{mzdepth}[0pt]{\the\dp\z@}
\newcount\c@@mz
\newcommand*{\the@mz}{mz\the\c@@mz}
\newcommand*{\@mz@list}{}
\let\@mz@do\relax
\newcommand*{\mzreset}{%
	\begingroup
	\def\@mz@do##1{%
		\global\expandafter\let\csname mz@##1\endcsname\relax
	}%
	\@mz@list
	\global\let\@mz@list\@empty
	\endgroup
}
\newcommand*{\mzleft}[3]{%
	\@ifundefined{mz@#1}{%
		\global\advance\c@@mz\@ne
		\expandafter\xdef\csname mz@#1\endcsname{\the@mz}%
		\xdef\@mz@list{\@mz@list\@mz@do{#1}}%
	}{}%
	\expandafter\let\expandafter\@mz\csname mz@#1\endcsname
	\mleft#2%
	\expandafter\mathpalette\expandafter{%
		\expandafter\@mzleft\expandafter{\@mz}%
	}{#3}%
	\mright.\kern-\nulldelimiterspace
}
\newcommand*{\mzright}[3]{%
	\kern-\nulldelimiterspace
	\@ifundefined{mz@#1}{%
		\@latex@warning{Missing \string\mzleft{#1}}%
		\mleft.#2\mright#3%
	}{%
		\expandafter\let\expandafter\@mz\csname mz@#1\endcsname
		\mleft.%
		\expandafter\mathpalette\expandafter{%
			\expandafter\@mzright\expandafter{\@mz}%
		}{#2}%
		\mright#3%
	}%
}
\newcommand*{\@mzleft}{%
	\@mzleftright lr%
}
\newcommand*{\@mzright}{%
	\@mzleftright rl%
}
\newcommand*{\@mzleftright}[5]{%
	\sbox0{$\m@th#4{}#5{}$}%
	\ifmeasuring@
	\else
	\begingroup
	\let\@auxout\@mainaux
	\zref@labelbyprops{#3#1}{mzheight,mzdepth}%
	\endgroup
	\fi
	\zifrefundefined{\@mz #2}{%
	}{%
		\dimen@=\zref@extract{#3#2}{mzheight}\relax
		\ifdim\dimen@>\ht0 %
		\ht0=\dimen@
		\fi
		\dimen@=\zref@extract{#3#2}{mzdepth}\relax
		\ifdim\dimen@>\dp0 %
		\dp0=\dimen@
		\fi
	}%
	\copy0\relax
}
\makeatother

\newcommand{\IEM}[2]{\llbracket #1,#2 \rrbracket}
\newcommand{\e}{\varepsilon}
\newcommand{\K}{\mathbb{K}}
\newcommand{\N}{\mathbb{N}}
\newcommand{\ssi}{if and only if}
\newcommand{\Y}{\raisebox{-0.25\height}{\begin{tikzpicture}[line cap=round,line join=round,>=triangle 45,x=0.2cm,y=0.2cm]
	\clip(-1.2,0.) rectangle (1.2,2.);
	\draw [line width=.5pt] (0.,0.)-- (0.,1.);
	\draw [line width=.5pt] (0.,1.)-- (-1.,2.);
	\draw [line width=.5pt] (0.,1.)-- (1.,2.);
	\end{tikzpicture}}}
\newcommand{\Yd}{\raisebox{-0.2\height}{\begin{tikzpicture}[line cap=round,line join=round,>=triangle 45,x=0.2cm,y=0.2cm]
	\draw [line width=.5pt] (0.,0.)-- (0.,1.);
	\draw [line width=.5pt] (0.,1.)-- (-1.,2.);
	\draw [line width=.5pt] (0.,1.)-- (1.,2.);
	\draw (1.,1.7) node[above]{\footnotesize{$d$}};
	\end{tikzpicture}}}
\newcommand{\Ydec}[1]{\raisebox{-0.3\height}{\begin{tikzpicture}[line cap=round,line join=round,>=triangle 45,x=0.25cm,y=0.25cm]
			\draw [line width=.5pt] (0.,0.)-- (0.,1.);
			\draw [line width=.5pt] (0.,1.)-- (-1.,2.);
			\draw [line width=.5pt] (0.,1.)-- (1.,2.);
			\draw (0,1) node[left]{\scriptsize{$#1$}};
\end{tikzpicture}}}

\newcommand{\YY}{\raisebox{-0.3\height}{\begin{tikzpicture}[line cap=round,line join=round,>=triangle 45,x=0.2cm,y=0.2cm]
			\draw (0,0)--(0,1);
			\draw (0,1)--(-2,3);
			\draw (-1.25,2.25)--(-0.5,3);
			\draw (1.25,2.25)--(0.5,3);
			\draw (0,1)--(2,3);
\end{tikzpicture}}}

\newcommand{\peignedroit}[3]{\raisebox{-0.5\height}{\begin{tikzpicture}[line cap=round,line join=round,>=triangle 45,x=0.25cm,y=0.25cm]
	\draw (0,0)--(0,1);
	\draw (0,1)--(-1,2) node[left]{$F_{#1}$};
	\draw (0,1)--(6,7);
	\draw (2,3)--(1,4) node[left]{$F_{#2}$};
	\draw (3.,4.5) node[rotate=45]{$\cdots$};
	\draw (5,6)--(4,7) node[left]{$F_{#3}$};
	\end{tikzpicture} }}
\newcommand{\peignegauche}[3]{\raisebox{-0.5\height}{\begin{tikzpicture}[line cap=round,line join=round,>=triangle 45,x=0.25cm,y=0.25cm]
	\draw (0,0)--(0,1);
	\draw (0,1)--(1,2) node[right]{$F_{#1}$};
	\draw (0,1)--(-6,7);
	\draw (-2,3)--(-1,4) node[right]{$F_{#2}$};
	\draw (-3.,4.5) node[rotate=135]{$\cdots$};
	\draw (-5,6)--(-4,7) node[right]{$F_{#3}$};
	\end{tikzpicture}}}
\newcommand{\balais}{\raisebox{-0.3\height}{\begin{tikzpicture}[line cap=round,line join=round,>=triangle 45,x=0.2cm,y=0.2cm]
	\draw (0,0)--(0,2);
	\draw (0,1)-- (-1,2);
	\draw (0,1)-- (1,2);
	\end{tikzpicture}}}
\newcommand{\balaisg}{\raisebox{-0.3\height}{\begin{tikzpicture}[line cap=round,line join=round,>=triangle 45,x=0.2cm,y=0.2cm]
	\draw (0,0)-- (0,1);
	\draw (0,1)-- (-1,2);
	\draw (-0.5,1.5)-- (0,2);
	\draw (0,1)-- (1,2);
	\end{tikzpicture}}}
\newcommand{\balaisd}{\raisebox{-0.3\height}{\begin{tikzpicture}[line cap=round,line join=round,>=triangle 45,x=0.2cm,y=0.2cm]
	\draw (0,0)-- (0,1);
	\draw (0,1)-- (-1,2);
	\draw (0.5,1.5)-- (0,2);
	\draw (0,1)-- (1,2);
	\end{tikzpicture}}}

\makeatletter
\newcommand{\mathleft}{\@fleqntrue\@mathmargin0pt}
\newcommand{\mathcenter}{\@fleqnfalse}
\makeatother

\DeclareMathOperator{\op}{op}
\DeclareMathOperator{\Prim}{Prim}

\DeclareMathOperator{\batc}{QSh}

\DeclareMathOperator{\Light}{Light}
\DeclareMathOperator{\Tridend}{Tridend}
\DeclareMathOperator{\TSym}{TSym}

\DeclareMathOperator{\PBT}{PBT}
\DeclareMathOperator{\nf}{nl}
\DeclareMathOperator{\coT}{coT}
\DeclareMathOperator{\Dend}{Dend}

\begin{document}
\allowdisplaybreaks

\newcommand{\arXivNumber}{2207.03839}

\renewcommand{\PaperNumber}{066}

\FirstPageHeading

\ShortArticleName{Tridendriform Structures}
	
\ArticleName{Tridendriform Structures}
	
\Author{Pierre CATOIRE}
	
\AuthorNameForHeading{P.~Catoire}
	
\Address{Universit\'e du Littoral C\^ote d'Opale, UR 2597 LMPA, Laboratoire de Math\'ematiques Pures\\ et Appliqu\'ees Joseph Liouville, F-62100 Calais, France}
\Email{\href{mailto:pierre.catoire@univ-littoral.fr}{pierre.catoire@univ-littoral.fr}}
		
\ArticleDates{Received November 10, 2022, in final form August 31, 2023; Published online September 15, 2023}
					
\Abstract{Inspired by the work of J-L.~Loday and M.~Ronco, we build free tridendriform algebras over reduced trees and we show that they have a coproduct satisfying some compatibilities with the tridendriform products. Its graded dual is the opposite bialgebra of TSym introduced by N.~Bergeron et al., which is described by the lightening splitting of a tree. In particular, we can split the product in three pieces and the coproduct in two pieces with Hopf compatibilities. We generate its codendriform primitives and count its coassociative primitives thanks to L.~Foissy's work.}
		
		
\Classification{16S10; 16T05; 16T10; 16W50; 17A30}

\section{Introduction}
In particular cases, the associative algebra structure on a space $A$ may be extended to a richer one. For example, in~\cite{bidend} and~\cite{euleridem}, the authors consider objects called \emph{dendriform algebras}. In~\cite{Scind}, they consider dendriform algebras and dipterous algebras. Studying such properties over the product of a bialgebra can give us information about its coproduct, if some convenient compatibilities are satisfied. We can also consider the dendriform or tridendriform cases (commutative or not) detailed in~\cite{Quasi-shuffle} under the terminology of shuffle and quasi-shuffle algebras.

 To be more precise, an algebra is dendriform if its product can be ``broken'' in two pieces denoted $\prec$, $\succ$ called respectively ``left product'' and ``right product'' such that these products verify, for all $a,b,c\in A$
\begin{gather*}
	(a\prec b)\prec c= a\prec (b*c), \\
	(a\succ b)\prec c= a\succ (b\prec c), \\
	a\succ (b\succ c)= (a*b)\succ c.
\end{gather*}
In particular, this implies that $*\coloneqq\prec+\succ$ is associative. The article~\cite{bidend} adds the (dual) notion of dendriform coalgebra and describe dendriform bialgebras.
In this document, we will bring our attention to the case \emph{tridendriform} detailed in~\cite{Trial} and~\cite{Tridend} with $\lambda=1$. This means that our product $*$ is split in three pieces denoted $\prec$, $\cdot$, $\succ$, respectively called ``left product'', ``middle product'' and ``right product'' satisfying the properties of Definition~\ref{def:tridend}.
We define the associative product of a tridendriform algebra as $*\coloneqq\prec+\cdot+\succ$. To study a tridendriform bialgebra $H$, it is convenient to define a tridendriform structure over $H\otimes H$.
 In order to talk about coassociativity later, we remind that $H\otimes(H\otimes H)$ and $(H\otimes H)\otimes H$ have the same tridendriform structure
\begin{customlemma}{\ref{lem:tritensass}}
	For any tridendriform algebras $A$, $B$ and $C$, we have $A\otimes(B\otimes C)=(A\otimes B)\otimes C$ as tridendriform algebras.
\end{customlemma}
To build tridendriform bialgebras, we will need to add units. The unique way to add units is described in Section~\ref{sec:Aug}, with an extension $\overline{\otimes}$ of the tensor product. We get

\begin{customlemma}{\ref{lem:tensassoaugm}}
	For $\overline{A}$, $\overline{B}$ and $\overline{C}$ three augmented tridendriform algebras, we have
	\[
	A\,\overline{\otimes}\, (B\,\overline{\otimes} \,C)=(A\, \overline{\otimes}\,B)\, \overline{\otimes}\,C \text{ as tridendriform algebras.}
	\]
\end{customlemma}

After those reminders about the tensor products of tridendriform algebras and augmented tridendriform algebras, we will study the combinatorics of these algebras.
For this we will use~\cite{Trial} where the free tridendriform algebra with one generator is described with reduced trees, i.e., trees where each node has at least $2$ sons (also known as Schr\"oder trees). This algebra is denoted by $\mathcal{A}$. In this paper, we will get a new non-inductive formula for the tridendriform operations. Here are some examples of computations:
\begin{alignat*}{3}
	&\Y \prec \balais=\raisebox{-0.5\height}{\begin{tikzpicture}[line cap=round,line join=round,>=triangle 45,x=0.2cm,y=0.2cm]
		\draw (0,0)--(0,1);
		\draw (0,1)--(-2,3);
		\draw (0,1)--(2,3);
		\draw (1,2)--(1,3);
		\draw (1,2)--(0,3);
	\end{tikzpicture}}, \qquad &&
\balaisd \prec \balaisg =\raisebox{-0.5\height}{\begin{tikzpicture}[line cap=round,line join=round,>=triangle 45,x=0.15cm,y=0.15cm]
		\draw (0,0)--(0,1);
		\draw (0,1)--(-3,4);
		\draw (0,1)--(3,4);
		\draw (1,2)--(-1,4);
		\draw (1.66,2.66)--(0.33,4);
		\draw (1,3.33)--(1.66,4);
\end{tikzpicture}}+
\raisebox{-0.5\height}{\begin{tikzpicture}[line cap=round,line join=round,>=triangle 45,x=0.15cm,y=0.15cm]
		\draw (0,0)--(0,1);
		\draw (0,1)--(-3,4);
		\draw (0,1)--(3,4);
		\draw (1,2)--(-1,4);
		\draw (1,2)--(1,3);
		\draw (1,3)--(2,4);
		\draw (1,3)--(0,4);
\end{tikzpicture}}
+\raisebox{-0.5\height}{\begin{tikzpicture}[line cap=round,line join=round,>=triangle 45,x=0.15cm,y=0.15cm]
		\draw (0,0)--(0,1);
		\draw (0,1)--(-3,4);
		\draw (0,1)--(3,4);
		\draw (0.66,1.66)--(-1.66,4);
		\draw (-0.41,2.81)--(0.81,4);
		\draw (0.1,3.31)--(-0.61,4);
\end{tikzpicture}}
+\raisebox{-0.5\height}{\begin{tikzpicture}[line cap=round,line join=round,>=triangle 45,x=0.15cm,y=0.15cm]
		\draw (0,0)--(0,1);
		\draw (0,1)--(-3,4);
		\draw (0,1)--(3,4);
		\draw (1,2)--(-1,4);
		\draw (0,3)--(0,4);
		\draw (0,3)--(1,4);
\end{tikzpicture}}
+\raisebox{-0.5\height}{\begin{tikzpicture}[line cap=round,line join=round,>=triangle 45,x=0.15cm,y=0.15cm]
		\draw (0,0)--(0,1);
		\draw (0,1)--(-3,4);
		\draw (0,1)--(3,4);
		\draw (1,2)--(-1,4);
		\draw (0.33,2.66) -- (1.66,4);
		\draw (-0.33,3.33) -- (0.33,4);
\end{tikzpicture}},
&\\ &\Y \cdot \balaisg=\raisebox{-0.5\height}{\begin{tikzpicture}[line cap=round,line join=round,>=triangle 45,x=0.2cm,y=0.2cm]
		\draw (0,0)--(0,2);
		\draw (0,2)--(-1,3);
		\draw (0,2)--(1,3);
		\draw (0,1)--(-2 ,3);
		\draw (0,1)--(2,3);
\end{tikzpicture}}, \qquad&&
 \balaisd \cdot \balaisg= \raisebox{-0.5\height}{\begin{tikzpicture}[line cap=round,line join=round,>=triangle 45,x=0.2cm,y=0.2cm]
\draw (0,0)--(0,2);
\draw (0,1)--(-2,3);
\draw (0,1)--(2,3);
\draw (0,2)--(-1,3);
\draw (0.5,2.5)--(0,3);
\draw (0,2)--(1,3);
\end{tikzpicture}}+\raisebox{-0.5\height}{\begin{tikzpicture}[line cap=round,line join=round,>=triangle 45,x=0.2cm,y=0.2cm]
\draw (0,0)--(0,2);
\draw (0,1)--(-2,3);
\draw (0,1)--(2,3);
\draw (0,2)--(-1,3);
\draw (0,2)--(0,3);
\draw (0,2)--(1,3);
\end{tikzpicture}}+\raisebox{-0.5\height}{\begin{tikzpicture}[line cap=round,line join=round,>=triangle 45,x=0.2cm,y=0.2cm]
\draw (0,0)--(0,2);
\draw (0,1)--(-2,3);
\draw (0,1)--(2,3);
\draw (0,2)--(-1,3);
\draw (-0.5,2.5)--(0,3);
\draw (0,2)--(1,3);
\end{tikzpicture}},
 &\\
 &\Y\succ \balaisd=\raisebox{-0.5\height}{\begin{tikzpicture}[line cap=round,line join=round,>=triangle 45,x=0.2cm,y=0.2cm]
 		\draw (0,0)--(0,1);
 		\draw (0,1)--(-2,3);
 		\draw (0,1)--(2,3);
 		\draw (1.25,2.25)--(0.5,3);
 		\draw (-1.25,2.25)--(-0.5,3);
 \end{tikzpicture}},\qquad&&
\balaisg \succ \balaisg=\raisebox{-0.5\height}{\begin{tikzpicture}[line cap=round,line join=round,>=triangle 45,x=0.2cm,y=0.2cm]
		\draw (0,0)--(0,1);
		\draw (0,1)--(-2,3);
		\draw (0,1)--(2,3);
		\draw (-0.66,1.66)--(0.66,3);
		\draw (-1.33,2.33)--(-0.66,3);
\end{tikzpicture}}+
\raisebox{-0.5\height}{\begin{tikzpicture}[line cap=round,line join=round,>=triangle 45,x=0.2cm,y=0.2cm]
		\draw (0,0)--(0,1);
		\draw (0,1)--(-2,3);
		\draw (0,1)--(2,3);
		\draw (1.25,2.25)--(0.5,3);
		\draw (-1.25,2.25)--(-0.5,3);
\end{tikzpicture}}+
\raisebox{-0.5\height}{\begin{tikzpicture}[line cap=round,line join=round,>=triangle 45,x=0.2cm,y=0.2cm]
		\draw (0,0)--(0,3);
		\draw (0,1)--(-2,3);
		\draw (0,1)--(2,3);
		\draw (1.25,2.25)--(0.5,3);
\end{tikzpicture}}.&
\end{alignat*}
We will see in Section~\ref{sec4} that these products are quasi-shuffles of the branches of the right comb representation of the left term with the branches of the left comb representation of the right term. Looking at $t$ as a right comb and $s$ as a left comb means we consider them as
		\begin{align*}
			&t=\peignedroit{1}{2}{k}
			\qquad \text{and} \qquad s=\peignegauche{k+1}{k+2}{k+l}
		\end{align*}
\begin{customdefi}{\ref{def:quasiaction}}
	Let $t$ and $s$ be two trees which respective right comb and left comb representations are given above.
	Let $k$, $l$ be the numbers of forests in the comb representation of~$t$ and~$s$. Let $\sigma$ be a $(k,l)$-quasi-shuffle which has for image $\IEM{1}{n}$. We denote $\sigma(t,s)$ the tree obtained the following way:
	\begin{enumerate}\itemsep=0pt
		\item We first consider the ladder with $n$ nodes
		\[
			\raisebox{-0.5\height}
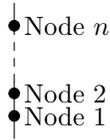
{\begin{tikzpicture}[line cap=round,line join=round,>=triangle 45,x=0.3cm,y=0.3cm]
					\begin{scope}
						\draw (0,0)--(0,2.5);
						\draw[dashed] (0,2.5)--(0,5);
						\draw (0,5)--(0,6);
						\filldraw [black] (0,1) circle (2pt) node[anchor=west]{\small{Node $1$}};
						\filldraw [black] (0,2) circle (2pt) node[anchor=west]{\small{Node $2$}};
						\filldraw [black] (0,5) circle (2pt) node[anchor=west]{\small{Node $n$}};
					\end{scope}
			\end{tikzpicture}}
		\]
		\item For all $i\in\IEM{1}{k},$ we graft $F_i$ as the \emph{left} son at the node $\sigma(i)$.
		\item For all $i\in\IEM{k+1}{k+l},$ we graft $F_i$ as \emph{right} son at the node $\sigma(i)$.
	\end{enumerate}
\end{customdefi}
\begin{customEg}{\ref{Eg:quasiaction}}
	Consider $\sigma=(1,3,2,3)$ a $(2,2)$-quasi-shuffle.

 Let us take
	$t=$	\raisebox{-0.35\height}{\begin{tikzpicture}[line cap=round,line join=round,>=triangle 45,x=0.2cm,y=0.2cm]
			\draw (0,0)--(0,1);
			\draw (0,1)--(-1,2) node[left]{$F_1$};
			\draw (0,1)--(2,3);
			\draw (1,2)--(0,3) node[left,above]{$F_2$};
	\end{tikzpicture}} and $s=$\nolinebreak \raisebox{-0.35\height}{\begin{tikzpicture}[line cap=round,line join=round,>=triangle 45,x=0.2cm,y=0.2cm]
			\draw (0,0)--(0,1);
			\draw (0,1)--(1,2) node[right]{$F_{3}$};
			\draw (0,1)--(-2,3);
			\draw (-1,2)--(0,3) node[right,above]{$F_{4}$};
	\end{tikzpicture}}.
	Then
	\[
	\sigma(t,s)=\text{\raisebox{-0.5\height}{\begin{tikzpicture}[line cap=round,line join=round,>=triangle 45,x=0.3cm,y=0.3cm]
			\draw (0,0)--(0,4);
			\draw (0,1)--(-1,2) node[left]{$F_1$};
			\draw (0,2)--(1,3) node[right]{$F_3$};
			\draw (0,3)--(-1,4) node[left]{$F_2$};
			\draw (0,3)--(1,4) node[right]{$F_{4_{\vphantom{\big|}}}$};
	\end{tikzpicture} }}.
	\]
\end{customEg}

\begin{customthm}{\ref{produit}}
	Let $t$, $s$ be two trees different from $|$. Then
	\[
	t*s=\sum_{\sigma\in \batc(k,l)}\sigma(t,s).
	\]
\end{customthm}
We also get a similar description for $\prec$, $\cdot$, $\succ$ products. For more details, see Corollary~\ref{Cor:3prodescription}. Then considering the number of leaves minus one as a gradation, we obtain
\begin{customthm}{\ref{coproduit}}
	The unique coproduct $\Delta$ over $\mathcal{A}$ compatible with the tridendriform structure and making $(\mathcal{A},*,|,\Delta,\varepsilon)$ a connected graded bialgebra is given by the following formula for all tree~$t$
	\begin{gather*}
		\Delta(t)=\sum_{c \text{ \normalfont{admissible cut of} }t} G^c(t)\otimes R^c(t), \qquad
		\Delta(|)=|\otimes |,
	\end{gather*}
	where $R^c(t)$ is the component which has the root of $t$ and $G^c(t)=G^c_1(t)*\cdots *G^c_k(t),$ where $G^c_i(t)$ are naturally ordered.
	If $c$ is the empty cut, we define $R^c(t):=t$ and $G^c(t):=|.$ If $c$ is the total cut, we define $R^c(t):=|$ and $G^c(t)=t.$	
\end{customthm}

We shall use the notation $(n,m)$-dendriform bialgebras for $n$, $m$ two non-negative integers, to talk about bialgebras where the product can be split in $n$ other operations and the coproduct can be broken in $m$ other operations, satisfying some Hopf compatibilities. Here is a table of known $(n,m)$-dendriform algebras, where $n$ is the line number and $m$ the column's one,\looseness=1
\begin{center}
	{\footnotesize\renewcommand{\arraystretch}{1.2}
		\begin{tabular}{|c|m{2.3cm}|m{2.3cm}|m{2.0cm}|m{2.0cm}|m{2.0cm}|}
			\hline
			$(n,m)$ & 0 & 1 & 2 & 3 & 4 \\
			\hline
			0 & $\K$-vector space & Associative \newline algebra & Dendriform\newline algebra & Tridendriform\newline algebra & Quadriform\newline algebra \\
			\hline
			1 & Coalgebra & Bialgebra & Dendriform\newline bialgebra & Tridendriform\newline bialgebra & Quadriform\newline bialgebra \\
			\hline
			2 & Codendriform\newline coalgebra & Codendriform\newline bialgebra & Bidendriform\newline bialgebra & $(3,2)$-\newline dendriform\newline bialgebra & ? \\
			\hline
			3 & Cotridendriform\newline coalgebra & Cotridendriform\newline bialgebra & $(2,3)$-\newline dendriform\newline bialgebra & See~\cite[Definition 5.16]{Tridend+co} & ? \\
			\hline
			4 & Coquadriform\newline coalgebra & Coquadriform\newline bialgebra & ? & ? & ? \\
			\hline
		\end{tabular}
	}
\end{center}
In this paper, we introduce the new definitions of $(3,2)$ and $(2,3)$-dendriform bialgebras. The reader will find some pieces of information about tridendriform algebras, cotridendriform coalgebras and cotridendriform tridendriform bialgebras in~\cite{Tridend+co} at Sections 5.1, 5.3 and~5.16, respectively. Tridendriform algebras are studied in~\cite{Tridend}. Bidendriform bialgebras are introduced and studied in~\cite{bidend}.

 The fourth section introduces the notion of $(1,3)$-dendriform bialgebras which is the dual notion of $(3,1)$-dendriform bialgebras. Thanks to the previous theorem, we give a link between the graded dual of $\mathcal{A}$, denoted by $\mathcal{A}^{\circledast}$, and the recent paper of N. Bergeron and al~\cite{Eclair}, where the Hopf algebra $\TSym$ is introduced
\begin{customthm}{\ref{thm:dual}}
	The bialgebras $\mathcal{A}^{\circledast}$ and $\TSym^{\op}$ are the same.
\end{customthm}
So, $\TSym$ is a $(1,3)$-dendriform bialgebra and we get a description of its graded dual.
Finally, in the last section, we introduce the definition of $(3,2)$-dendriform bialgebras
 \begin{customdefi}{\ref{deftroisdeux}}\samepage
	A \emph{$(3,2)$-dendriform bialgebra} is a sextuple $(A,\prec,\cdot,\succ,\Delta_{\leftarrow},\Delta_{\rightarrow})$ such that
	\begin{itemize}\itemsep=0pt
		\item $(A,\prec,\cdot,\succ)$ is an augmented tridendriform algebra.
		\item $(A,\Delta_{\leftarrow},\Delta_{\rightarrow})$ is an augmented dendriform coalgebra.
		\item Using generalized Sweedler's notations, the following relations are true, for all $a,b\in A$
		\begin{gather*}
			\Delta_{\leftarrow}(a\cdot b) =a'b'_{\leftarrow}\otimes a''\cdot b''_{\leftarrow}+b'_{\leftarrow}\otimes a\cdot b''_{\leftarrow},\\
			\Delta_{\rightarrow}(a\cdot b) =a'b'_{\rightarrow}\otimes a''\cdot b''_{\rightarrow}+a'\otimes a''\cdot b+b'_{\rightarrow}\otimes a\cdot b''_{\rightarrow}, \\
			\Delta_{\leftarrow}(a\prec b) =a'b'_{\leftarrow}\otimes a''\prec b''_{\leftarrow}+a'b\otimes a'' +b'_{\leftarrow}\otimes a\prec b''_{\leftarrow}+b\otimes a,\\
			\Delta_{\rightarrow}(a\prec b) =a'b'_{\rightarrow}\otimes a''\prec b''_{\rightarrow}+a'\otimes a''\prec b +b'_{\rightarrow}\otimes a\prec b''_{\rightarrow},\\
			\Delta_{\rightarrow}(a\succ b) =a'b'_{\rightarrow}\otimes a''\succ b''_{\rightarrow}+a'\otimes a''\succ b +b'_{\rightarrow}\otimes a\succ b''_{\rightarrow}+ab'_{\rightarrow}\otimes b''_{\rightarrow}+a\otimes b, \\
			\Delta_{\leftarrow}(a\succ b) =a'b'_{\leftarrow}\otimes a''\succ b''_{\leftarrow}+a'\otimes a''\succ b +ab'_{\leftarrow}\otimes b''_{\leftarrow}.
		\end{gather*}\mathcenter
	\end{itemize}
\end{customdefi}
We give an example of such an object.
\begin{customprop}{\ref{Prop:troideuxEg}}
	We consider $(\mathcal{A},\prec,\cdot,\succ)$ with its tridendriform algebra structure, we build the following coproducts:
	\begin{gather*}
		\tilde{\Delta}_{\leftarrow}(t)=\left(\sum_{\substack{\text{\rm $c$ admissible cut of $t$} \\ \text{\rm the right leaf of $t$ has been cut}}} G^c(t)\otimes R^c(t)\right)-t\otimes 1, \\
		\tilde{\Delta}_{\rightarrow}(t)=\left(\sum_{\substack{\text{\rm $c$ admissible cut of $t$} \\ \text{\rm the right leaf of $t$ has not been cut}}} G^c(t)\otimes R^c(t)\right)-1\otimes t.
	\end{gather*}
Then $(\mathcal{A},\prec+\cdot,\succ,\Delta_{\leftarrow},\Delta_{\rightarrow})$ and $(\mathcal{A},\prec,\cdot+\succ,\Delta_{\leftarrow},\Delta_{\rightarrow})$ are bidendriform bialgebras. As a~consequence, $(\mathcal{A},\prec,\cdot,\succ,\Delta_{\leftarrow},\Delta_{\rightarrow})$ is a $(3,2)$-dendriform bialgebra.
\end{customprop}
There, we can apply the results from~\cite{bidend} to count and describe the codendriform primitives ans coassociative primitives. With this work, we are able to get all the codendriform primitives and we can count the coassociative ones. Finally, to end the paper, we study the ``natural'' quotient of $\mathcal{A}$ as a $(3,2)$-dendriform algebra by the set of elements obtained with the middle product $\cdot$ (we send the middle product $\cdot$ to $0$) in order to get back to the Loday--Ronco algebra given in~\cite{LR} and~\cite{AgLR}.

\subsection*{Some details about the bibliography}

 For other examples of tridendriform algebras, we can refer to the quasi-shuffle algebras present in~\cite{batc2, batc} or~\cite{batc3}. The tridendriform algebra of the parking functions is given in~\cite{Tridend}.
 In the document, Remarks \ref{Rq:vecsapcemult}, \ref{Rq:multcoprod} and \ref{Rq:multprod} describe the case with multiple generators. They are given without proofs, as it follows from the proofs detailed for one generator.

\subsection*{Document structure}

In all the document, $\K$ will be any commutative field. This document has four sections:
\begin{itemize}\itemsep=0pt
	\item The first one gives several reminders about tridendriform algebras and its augmentations.
	\item The following one describes the free tridendriform bialgebra (relying on~\cite{Trial}) with one generator, giving combinatorial descriptions of the product with quasi-shuffle and admissible cuts. There is also a description for multiple generators given as remarks.
	\item Then we study the graded dual of the free tridendriform bialgebra which is $\TSym$, the bialgebra whose coproduct is given by the lightning decomposition of a tree. As a consequence of this section, we will find a $(1,3)$-dendriform bialgebra structure for $\TSym$.
	\item The last section defines $(3,2)$-dendriform bialgebras and describe one example built on the free $(3,1)$-dendriform bialgebra with one generator given in the third section. We get it by decorating the right-most leaf of a tree and we check it satisfies the definition. Finally, we will show that a quotient of the free $(3,2)$-dendriform algebra is the Loday--Ronco algebra.
\end{itemize}

\section{Reminder}
\subsection{Tensor product for tridendriform algebras} \label{sec:tensor}
\begin{defi}\label{def:tridend}
	Let $\K$ be a field and $A$ be vector space of $\K$. We say $(A,\prec,\cdot,\succ)$ is a \emph{tridendriform algebra} if $\prec$, $\cdot$ and $\succ$ are all linear maps from $A\otimes A$ to $A$ such that for any $a,b,c\in A$
	\begin{gather}
		(a\prec b)\prec c =a\prec(b*c), \label{tri1}\\
		(a\succ b)\prec c =a\succ(b\prec c), \label{tri2}\\
		(a* b)\succ c =a\succ(b\succ c), \label{tri3} \\
		(a\succ b)\cdot c =a\succ(b\cdot c), \label{tri4} \\
		(a\prec b)\cdot c =a\cdot(b\succ c), \label{tri5} \\
		(a\cdot b)\prec c =a\cdot(b\prec c), \label{tri6} \\
		(a\cdot b)\cdot c =a\cdot (b\cdot c), \label{tri7}
	\end{gather}
where $*$ is the \emph{associative product} of the tridendriform algebra defined for any $a,b\in A$ by $a*b\coloneqq a\prec b + a\cdot b +a\succ b$.
We call the products $\succ$, $\prec$, $\cdot$ respectively right, left and middle.
\end{defi}
\begin{Rq}
	The associativity of $*$ is a consequence of those seven relations.
\end{Rq}
 Let us consider an associative algebra $(A,\star)$ and a tridendriform algebra $(B,\succ,\cdot,\prec)$.
 Our goal is to give $A\otimes B$ a tridendriform algebra structure. In order to reach this aim, let us define for all $\ltimes\in\{*,\prec,\succ,\cdot\}$
\begin{equation*}
\ltimes\colon \ \begin{cases}
(A\otimes B)\otimes (A\otimes B) \rightarrow A\otimes B, \\
(a\otimes b)\otimes (c\otimes d)\mapsto a \star c\otimes b\ltimes d,
\end{cases}
\end{equation*}
 With these definitions, we have $*=\prec +\cdot +\succ$.

 \begin{Lemme}\label{lem:tritens}
 	$(A\otimes B,\prec,\cdot,\succ)$ is a tridendriform structure.
 \end{Lemme}
\begin{proof}
 We wish that these definitions make $A\otimes B$ a tridendriform algebra.
 We will check the relations $\eqref{tri1}$ to $\eqref{tri7}$.
 To shorten the proof, we notice that the relations can be written as follows for all $a,b,c\in A$
 \begin{equation*}
 	(a\ltimes b)\rtimes c = a\ltimes' (b\rtimes' c),
 \end{equation*}
 for $(\ltimes, \rtimes, \ltimes', \rtimes')$ belonging to the following set:
 \begin{equation} 	
 	\{ (\prec,\prec, \prec,*),(\succ,\prec,\succ,\prec),(*,\succ,\succ,\succ),(\succ,\cdot,\succ,\cdot),(\prec,\cdot, \cdot,\succ),(\cdot, \prec,\cdot, \prec),(\cdot,\cdot,\cdot,\cdot) \}. \label{quadru}
 \end{equation}
 Let $(\ltimes,\rtimes,\ltimes',\rtimes')$ be an element of the set above. Let $a,c,e\in A$ and $b,d,f\in B$, then
 \begin{align*}
 	((a\otimes b) \ltimes (c\otimes d))\rtimes e\otimes f &=(a\star c \otimes b\ltimes d) \rtimes (e\otimes f)
 												 =(a\star c)\star e \otimes (b\ltimes d) \rtimes f \\
 												 &= a\star (c\star e) \otimes b\ltimes'(d \rtimes' f)
 												 = (a\otimes b)\ltimes'((c\otimes d) \rtimes' (e\otimes f)).
 \end{align*}
 Consequently, this shows that $A\otimes B$ is a tridendriform algebra.
 \end{proof}

Moreover, the tensor product for tridendriform algebras is associative.
 \begin{Lemme}\label{lem:tritensass}
 	For any tridendriform algebras $A$, $B$ and $C$, we have $A\otimes(B\otimes C)=(A\otimes B)\otimes C$ as tridendriform algebras.
 \end{Lemme}
\begin{proof}
One just needs to write down the definitions of the products.
 \end{proof}

\subsection{Augmented tridendriform algebras }\label{sec:Aug}
Let $(A,\succ, \prec, \cdot)$ be a tridendriform algebra. If we do some easy computations, we find that $(A,\succ +\cdot, \prec)$ and $(A,\succ, \prec +\cdot)$ are two dendriform algebras, as detailed in $\cite{bidend}$.
In order to define consistently the concept of \emph{augmented tridendriform algebra}, which we will denote by $(\overline{A},\succ,\prec,\cdot)$ where $\overline{A}=A\oplus \K1$ ($1$ is a unit for $*$ that we formally add to $A$), we need that $(\overline{A},\succ +\cdot, \prec)$ and $(\overline{A},\succ, \prec +\cdot)$ are both augmented dendriform algebras as defined in~\cite[Section~3.1]{Augm}.
Therefore, the definition of dendriform algebras applied to $(\overline{A},\succ +\cdot, \prec)$ gives us for all $a\in A$
\begin{equation*}
	a\prec 1+ a\cdot 1=a=1\succ a \qquad \text{and} \qquad 1\prec a + 1 \cdot a=0=a\succ 1.
\end{equation*}

Applying the same definition to $(\overline{A},\succ, \prec+\cdot)$, we get for all $a\in A$
\begin{equation*}
a\prec 1 =a=1\succ a+1\cdot a \qquad \text{and} \qquad 1\prec a =0=a\succ 1+a\cdot 1.
\end{equation*}

As a consequence, we need to define $1\cdot a=a\cdot 1=0$ for all $a\in A$.

\begin{defi}
	Let $(A,\succ,\prec,\cdot)$ be a tridendriform algebra. We expand the products on \linebreak$\K\otimes A\oplus A\otimes \K \oplus A\otimes A$ by defining
	\begin{gather}\label{def:augment}
		\forall a\in A, \!\qquad 1\prec a :=0=1\succ a, \!\qquad a\prec 1:=a=1\succ a \!\qquad \text{and} \!\qquad a\cdot 1:=0=1\cdot a.
	\end{gather}
	This construction gives us what we will call \emph{augmented tridendriform algebra}. For a tridendriform algebra $A$, we denote $\overline{A}$ the associated augmented tridendriform algebra.
\end{defi}

\begin{Rq}\label{Rq:relations+}
	With this definition of augmented tridendriform algebra, when both terms of the relations~\eqref{tri1}--\eqref{tri7} are well defined, there is an equality.
\end{Rq}

\begin{Rq}
	This is the unique way to define $a\succ 1$, $1\succ a$, $a\prec 1$ and $1\prec a$ such that~\eqref{tri1}--\eqref{tri7} are equalities as soon as both sides are well defined.
\end{Rq}
\begin{Rq}
	It is not possible to define $1\prec 1$ and $1\succ 1$ such that the relations~\eqref{def:augment} are satisfied over $\overline{A}.$
\end{Rq}

\subsection{Tensor product for augmented tridendriform algebras}

Let $\big(\overline{A},\succ,\prec,\cdot\big)$ and $\big(\overline{B},\succ,\prec,\cdot\big)$ be two augmented tridendriform algebras. As in the dendriform case, it is impossible to define $\succ$ and $\prec$ on the pair $(1,1)$ such that~\eqref{def:augment} still occurs. For this reason, we define
\[
A\,\overline{\otimes}\, B:=A\otimes B\oplus \K\otimes B \oplus A\otimes \K.
\]
The previous vector space is a subspace of $\overline{A}\otimes \overline{B}$ on which we will define the three opera\-tions~$\succ$,~$\prec$ and~$\cdot$.
\begin{Rq}\qquad
	\begin{itemize}\itemsep=0pt
		\item We can extend $*$ to $\overline{A}\otimes \overline{A}$ with $1*1=1$. Moreover $*$ is associative on $\overline{A}$.
		\item For two tridendriform algebras $A$ and $B$, we have $\overline{A\,\overline{\otimes}\, B}=\overline{A}\otimes \overline{B}$.
	\end{itemize}
\end{Rq}
We define
\begin{equation*}
	*\colon \ \begin{cases}
	\big(\overline{A}\otimes \overline{B}\big)^{\otimes 2} \rightarrow \big(\overline{A}\otimes \overline{B}\big), \\
	a\otimes b\otimes c\otimes d \mapsto a*c\otimes b*d.
	\end{cases}
\end{equation*}
To simplify the notations, we denote by
\begin{gather*}
	D=\left((A\otimes \K)^{\otimes 2}\right)\oplus\left((\K \otimes \K)\otimes (A\otimes \K)\right)\oplus \left((A\otimes \K)\otimes (\K\otimes \K)\right), \\
	 U= \left((A\otimes B)^{\otimes 2}\right)\oplus \left((A\otimes \K)\otimes (A\otimes B)\right)\oplus \left((A\otimes B)\otimes (A\otimes \K)\right).
\end{gather*}
So, $(A\overline{\otimes} B)^{\overline{\otimes} 2}=U\oplus D$. Then for all $\ltimes \in\{\prec,\cdot,\succ\}$, we define
\begin{align*}
&\ltimes\colon \ \begin{cases}
 (A\,\overline{\otimes}\, B)^{\overline{\otimes} 2} \rightarrow \big(\overline{A}\otimes \overline{B}\big), \\
a\otimes b\otimes c\otimes d\in U \mapsto a*c\otimes b\ltimes d, \\
a\otimes b\otimes c\otimes d \in D \mapsto a\ltimes c \otimes b*d.
\end{cases}
\end{align*}

\begin{Lemme}
	We consider two tridendriform algebras $A$ and $B$. The algebra $A\,\overline{\otimes}\, B$ is a tridendriform algebra with the products defined above.
\end{Lemme}
\begin{Rq}
	This implies that $\overline{A}\otimes\overline{B}$ is an augmented tridendriform algebra.
\end{Rq}
\begin{proof}
In Section~\ref{sec:tensor}, we have checked that the relations~\eqref{tri1}--\eqref{tri7} are true when~$a$,~$b$,~$c$,~$d$,~$e$,~$f$ are elements of $A$ or $B$. We now need to check those relations when at least one of these elements is a unit.

Let $(\ltimes,\rtimes,\ltimes',\rtimes')$ be an element of the set defined at $\eqref{quadru}$.
To simplify the following, when we will consider $x\in\K 1$, we will suppose $x=1$ after multiplying by a scalar. There are four cases to check:
\begin{itemize}\itemsep=0pt
	\item \textit{Case} 1: $d\in B$ or $(b\in B \text{ and } f\in B)$;
	\item \textit{Case} 2: $b,d,f\in 1_B\K$;
	\item \textit{Case} 3: $d\in 1_B\K$, $f\in 1_B\K$ and $b\in B$;
	\item\textit{Case} 4: $b\in1_B\K$, $d\in1_B\K$, $f\in B$.
\end{itemize}
We prove Case~3 as an example. Other cases are left to the reader. We suppose $d\in 1_B\K$, $f\in 1_B\K$ and $b\in B$.
Then $c$ and $e$ cannot be units.
\begin{gather*}
	[(a\otimes b)\ltimes (c\otimes 1)]\rtimes (e\otimes 1) =(a*c\otimes b\ltimes 1)\rtimes (e\otimes 1)
	 =(a*c)*e\otimes (b\ltimes 1)\rtimes 1 \\
\hphantom{[(a\otimes b)\ltimes (c\otimes 1)]\rtimes (e\otimes 1)}{}=a*(c*e)\otimes b \times \delta_{(\ltimes,\rtimes),(\prec,\prec)}, \\
	(a\otimes b)\ltimes'[(c\otimes 1)\rtimes' (e\otimes 1)]
	 = a\otimes b\ltimes'(c\rtimes' e\otimes 1)
	 =a*(c\rtimes' e)\otimes b\ltimes' 1 \\
	\hphantom{(a\otimes b)\ltimes'[(c\otimes 1)\rtimes' (e\otimes 1)]}{} =a*(c\rtimes' e)\otimes b\ltimes' 1 \times \delta_{(\ltimes',\rtimes'),(\prec,*)}.
\end{gather*}
Having a look at $\eqref{quadru}$, we see that the only quadruple with $(\ltimes,\rtimes)=(\prec,\prec)$ is $(\prec,\prec,\prec,*)$. So, there is equality in this case.
Once, we have checked the relations~\eqref{tri1}--\eqref{tri7} on $A\,\overline{\otimes}\, B$.
There is still to check $*=\prec +\cdot +\succ$. With the work done in the first section, we only need to compute this when there are at least one unit. Those verifications are straightforward.
\end{proof}

\subsection[Associativity of overline\{otimes\}]{Associativity of $\boldsymbol{\overline{\otimes}}$}

Given $\overline{A}$, $\overline{B}$ and $\overline{C}$ three augmented tridendriform algebras, do $(A\,\overline{\otimes}\, B)\,\overline{\otimes}\, C$ and $A\,\overline{\otimes}\,(B\,\overline{\otimes}\,C)$ have the same tridendriform algebra structure?

From the first section, $(A\otimes B)\otimes C$ and $A\otimes (B\otimes C)$ have the same tridendriform algebra structure. There is still to check what happens with units.

\begin{Lemme}\label{lem:tensassoaugm}
	For $\overline{A}$, $\overline{B}$ and $\overline{C}$ three augmented tridendriform algebras, we have
	\[
	A\,\overline{\otimes}\,(B\,\overline{\otimes}\, C)=(A\,\overline{\otimes}\,B)\,\overline{\otimes}\,C \text{ as tridendriform algebras.}
	\]
\end{Lemme}

\begin{proof}
In the following, we will use the identifications $A\otimes \K\approx A$, $B\otimes \K\approx B$, $C\otimes \K\approx C$ and $A\otimes(B\otimes C)\approx (A\otimes B)\otimes C$. As a consequence, we see that $\big(\overline{A}\otimes \overline{B}\big)\otimes \overline{C}$ and $\overline{A}\otimes \big(\overline{B}\otimes \overline{C}\big)$ have the same underlying vector space.
 We put $D$ the following set:
 \[
 \big(\overline{A} \otimes \overline{B}\otimes C\big)^{\otimes 2} \oplus\big(\big(\overline{A}\otimes \overline{B}\otimes C\big)\otimes \big(\overline{A}\otimes \overline{B}\otimes \K\big)\big) \oplus \big(\big(\overline{A}\otimes \overline{B}\otimes \K\big)\otimes \big(\overline{A}\otimes \overline{B}\otimes C\big)\big).
 \]
 For any $\ltimes\in\{*,\prec,\succ,\cdot\}$, the tridendriform structure over $((A \,\overline{\otimes}\, B)\,\overline{\otimes}\, C)^{\overline{\otimes} 2}$ is given by
	\begin{align*}
	&\ltimes_{1}\colon \ \begin{cases}
	 ((A \,\overline{\otimes}\, B)\,\overline{\otimes}\, C)^{\overline{\otimes} 2} \rightarrow (A\,\overline{\otimes}\,B)\,\overline{\otimes}\, C ,\\
	(a\otimes b)\otimes c \otimes (a'\otimes b')\otimes c'\in D \mapsto (a\otimes b)*(a'\otimes b')\otimes c\ltimes c', \\
	(a\otimes b)\otimes c \otimes (a'\otimes b')\otimes c'\in ((A\overline{\otimes} B)\otimes \K)^{\overline{\otimes} 2} \mapsto (a\otimes b)\ltimes(a'\otimes b')\otimes c* c'.
	\end{cases}
\end{align*}
Extending by linearity, we get an map $\ltimes_{1}$ defined on all $((A\,\overline{\otimes}\, B)\,\overline{\otimes}\, C)^{\overline{\otimes} 2}.$
Let us put $D'$ the set
\[
\big(\overline{A}\otimes (B\,\overline{\otimes}\,C)\big)^{\overline{\otimes} 2}\oplus \big(\big(\overline{A}\otimes \K^{\otimes 2}\big)\otimes \big(\overline{A}\otimes (B\,\overline{\otimes}\,C)\big)\big)\oplus \big(\big(\overline{A}\otimes (B\,\overline{\otimes}\,C)\big)\otimes \big(\overline{A}\otimes \K^{\otimes 2}\big)\big).
\]
The tridendriform structure on $\overline{A} \otimes \big(\overline{B}\otimes \overline{C}\big)$ is given for all $\ltimes\in \{ *,\prec,\succ,\cdot \}$ by
\begin{equation*}
	\ltimes_2\colon \ \begin{cases}
	 (A\,\overline{\otimes}\, (B\,\overline{\otimes}\, C))^{\overline{\otimes} 2} \rightarrow A\,\overline{\otimes}\,(B\,\overline{\otimes}\, C), \\
	a\otimes (b\otimes c) \otimes a'\otimes(b'\otimes c')\in D' \mapsto 	a*a'\otimes [(b\otimes c)\ltimes (b'\otimes c')], \\
	a\otimes (b\otimes c) \otimes a'\otimes(b'\otimes c')\in\smash{\big(A\otimes\K^{\otimes 2 }\big)}^{\overline{\otimes} 2} \mapsto 	a\ltimes a'\otimes [(b\otimes c)* (b'\otimes c')].
	\end{cases}
\end{equation*}
By linearity of $\ltimes_1$ and $\ltimes_2$, for all $x\in \overline{A}$.
Then doing some straightforward computations, we obtain for all $\ltimes\in\{\prec,\succ,\cdot,*\}$, $\ltimes_1=\ltimes_2$. This proves the lemma.
\end{proof}

\begin{Not}
	Let $A$ be an augmented tridendriform algebra. We denote by $A^+$ the augmentation ideal of $A$.
\end{Not}
 \begin{Rq}
 	With those notations, we have $\overline{A^+}=A$.
 \end{Rq}
 Before ending this section, we will need a definition and an easy lemma.
 \begin{defi}
 	Let $A$, $B$ be two tridendriform algebras. Let $f\colon A\rightarrow B$ a linear map.
 	We say that $f$ is a \emph{tridendriform algebra morphism} if for all $a,b\in A$ and for all $\ltimes\in\{\prec,\cdot,\succ\}$
 	\begin{align*}
 		f(a\ltimes b)&=f(a)\ltimes f(b).
 	\end{align*}
 	If $A$, $B$ are augmented, we say that $f$ is an \emph{augmented tridendriform algebra morphism} if $f\colon A^+\rightarrow B^+$ is a tridendriform algebra morphism with $f(1)=1.$
 \end{defi}
 \begin{Lemme}\label{Lemme:morphtridend}
 	Let $A$, $A'$, $B$, $B'$ be four augmented tridendriform algebras. Consider two tridendriform algebra morphisms $f\colon A\rightarrow A'$ and $g\colon B\rightarrow B'$. Then the map
 	\begin{equation*}
 		f\overline{\otimes} g \colon \ \begin{cases}
 			A^+\,\overline{\otimes}\, B^+ \rightarrow A'\,\overline{\otimes}\, B', \\
 			a\otimes b \mapsto f(a)\otimes g(b)
 		\end{cases}
 	\end{equation*}
 	is a tridendriform algebra morphism. Defining
 	\begin{equation*}
 		f\otimes g \colon \ \begin{cases}
 			A\otimes B \rightarrow A\otimes B, \\
 			a\otimes b\in A^+\,\overline{\otimes}\, B^+ \mapsto (f\,\overline{\otimes}\, g)(a\otimes b), \\
 			1\otimes 1 \mapsto 1\otimes 1,
 		\end{cases}
 	\end{equation*}
 	we get an augmented tridendriform algebra morphism.
 \end{Lemme}
\begin{proof}
	The proof is straightforward.
\end{proof}

\section{The free tridendriform algebra and its coproduct}

Let $\Tridend$ be the quadratic operad generated by $\Tridend_2=\langle\succ,\prec,\cdot\rangle $ satisfying the relations~\eqref{tri1}--\eqref{tri7}.
Let us consider $\mathcal{A}^+=\Tridend(\K)$. In other words, $\mathcal{A}^+$ is a free tridendriform algebra generated by one element. Referring to~\cite[Theorem 2.6]{Trial}, we get $\mathcal{A}^+=\bigoplus_{n\geq 1} \K T_n$ where~$T_n$ is the set of all planar rooted trees with $n+1$ leaves in which every internal node has fertility at least $2$. In other words, $T_n$ is the set of reduced trees (also called Schr\"oder trees) with~$n+1$ leaves. Then we add a unit to this tridendriform algebra. We will denote in the following $\mathcal{A}=\overline{\mathcal{A}^+}$,
\[
\mathcal{A}=\bigoplus_{n\geq 0} \K T_n.
\]
\begin{Egs}
	$T_0=\left\lbrace |\right\rbrace$, $T_1=\left\lbrace\Y\right\rbrace$,
		$T_2=\left\lbrace
		\balais,
		 \balaisg ,
		\balaisd
	 \right\rbrace$.
\end{Egs}

\begin{Rq} \label{Rq:vecsapcemult} \quad
		\begin{itemize}\itemsep=0pt
			\item For more generators, we need to define:
				Let $n\in\N$. Let $\mathcal{D}$ be a graded set by $\N$, this means $\mathcal{D}=\bigsqcup_{i\in\N} \mathcal{D}_i$. We denote $T^g_n(\mathcal{D})$ the set of all rooted trees with $n+1$ leaves such that each \emph{internal node} $x$ is decorated by an element of $\mathcal{D}_n$ where $n$ is the number of children of the internal node $x$.
			\item Let $X$ be a non-empty set. The vector space on which we define the free tridendriform algebra generated by $X$ is
			\[
			\mathcal{A}(X)=\bigoplus_{n=0}^{+\infty} T^g_n(X^\star),
			\]
			where $X^\star$ is the set of all finite words on the alphabet $X$ and for all $n\in\N$, $X^{\star}_n$ is the set of all words of length $n$ over the alphabet $X$.
			An equivalent definition exists decorating the angular sectors of a tree by one letter from $X$ instead. For more details, refer to~\cite[Section 4]{Tridendmult}. For example, this is two representations of a same element of $\mathcal{A}(X)$:
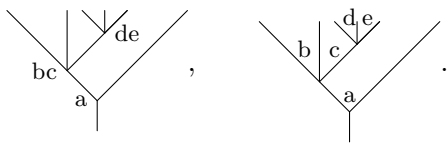
			\begin{gather*}
				\raisebox{-0.5\height}{\begin{tikzpicture}[line cap=round,line join=round,>=triangle 45,x=0.4cm,y=0.4cm]
						\draw (0,0)--(0,1);
						\draw (0,1)--(3,4);
						\draw (0,1)--(-3,4);
						\draw (-1,2)--(-1,4);
						\draw (-1,2)--(1,4);
						\draw (0.25,3.25)--(0.25,4);
						\draw (0.25,3.25)--(-0.5,4);
						\draw (0.25,3.25)--(1,4);
						\draw (0,1) node[below,left] {\footnotesize{a}};
						\draw (-1,2) node[left] {\footnotesize{bc}};
						\draw (0.25,3.25) node[right] {\footnotesize{de}};
				\end{tikzpicture}}, \qquad \raisebox{-0.5\height}{\begin{tikzpicture}[line cap=round,line join=round,>=triangle 45,x=0.4cm,y=0.4cm]
				\draw (0,0)--(0,1);
				\draw (0,1)--(3,4);
				\draw (0,1)--(-3,4);
				\draw (-1,2)--(-1,4);
				\draw (-1,2)--(1,4);
				\draw (0.25,3.25)--(0.25,4);
				\draw (0.25,3.25)--(-0.5,4);
				\draw (0.25,3.25)--(1,4);
				\draw (0,1) node[above] {\footnotesize{a}};
				\draw (-1.5,2.5) node[above] {\footnotesize{b}};
				\draw (-0.5,2.5) node[above] {\footnotesize{c}};
				\draw (0.6,3.55) node[above] {\footnotesize{e}};
				\draw (0,3.55) node[above] {\footnotesize{d}};
		\end{tikzpicture}}.
				\end{gather*}
		\end{itemize}
\end{Rq}

Let $k$ be a non-negative integer. Let $x^{(0)},\dots,x^{(k)}$ be trees from $T_{n_i}$ for $i\in\IEM{1}{k}$ and $n_i\in\N$. We denote by $x^{(0)}\vee \dots \vee x^{(k)}$ the tree obtained by grafting on the same root and from left to right the trees $x^{(0)},\dots, x^{(k)}$. Then for $x$ and $y$ two trees, there exist unique trees $x^{(0)},\dots,x^{(k)}$ and $y^{(0)},\dots, y^{(l)}$ such that
\begin{align*}
	&x=x^{(0)}\vee \dots \vee x^{(k)}, \qquad
	y=y^{(0)}\vee \dots \vee y^{(l)}.
\end{align*}

With those notations, Theorem 2.6 from~\cite{Trial} tells us that we can define the three products inductively by
\begin{gather*}
	x\prec y =x^{(0)}\vee \dots \vee \big(x^{(k)}*y\big), \\
	x\cdot y =x^{(0)}\vee \dots \vee \big(x^{(k)}*y^{(0)}\big)\vee \dots \vee y^{(l)}, \\
	x\succ y =\big(x*y^{(0)}\big)\vee \dots \vee y^{(l)},
\end{gather*}
 initializing the induction with $|\prec y=0$, $y\prec |=y$, $x\prec |=x$, $x\succ |=0$, $|\cdot y=0=|\cdot|$ and~$|*|=|$.
 \begin{Egs}\quad
 	\begin{itemize}\itemsep=0pt
 \item Taking $k=2$ and $x^{(0)}=|$, $x^{(1)}=\Y$, $x^{(2)}=|$, $x^{(0)}\vee x^{(1)}\vee x^{(2)}$ is the following tree:
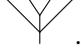
\[
 	\begin{tikzpicture}[line cap=round,line join=round,>=triangle 45,x=0.3cm,y=0.3cm]
 	\draw (0,0)-- (0,0.5);
 	\draw (0,0.5)-- (-1.5,2);
 	\draw (0,0.5)-- (1.5,2);
 	\draw (0,0.5)-- (0,1.25);
 	\draw (0,1.25)-- (-0.5,2);
 	\draw (0,1.25)-- (0.5,2);
 	\end{tikzpicture}.
 \]

\item On the other hand, the tree denoted by $x^{(1)}\vee x^{(0)}\vee x^{(2)}$ is
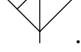
\[
	\begin{tikzpicture}[line cap=round,line join=round,>=triangle 45,x=0.3cm,y=0.3cm]
	\draw (0,0)-- (0,2);
	\draw (0,0.5)-- (-1.5,2);
	\draw (-1,1.5)-- (-0.5,2);
	\draw (0,0.5)-- (1.5,2);
	\end{tikzpicture}.
\]
\end{itemize}
\end{Egs}
Our investigation to find a bialgebra with a tridendriform structure begins here. With this objective in mind, we need to endow $\mathcal{A}$ with a bialgebra structure. Using the dendriform structure of $\mathcal{A}\otimes \mathcal{A}$, it is natural to impose for all $x,y\in \mathcal{A}$
\begin{gather*}
	\Delta(x\prec y)=\Delta(x)\prec \Delta(y),\qquad \Delta(x\cdot y)=\Delta(x)\cdot\Delta(y),\qquad \Delta(x\succ y)=\Delta(x)\succ\Delta(y),	
\end{gather*}
which leads to the following definition:
\begin{defi}
	We will call a \emph{tridendriform bialgebra} or \emph{$(3,1)$-dendriform bialgebra} any augmented tridendriform algebra $(H,\succ,\prec,\cdot)$ which also has a bialgebra structure $(H,*,1,\Delta,\varepsilon)$, where $*=\prec+\cdot+\succ$, satisfying the following compatibilities, for all $x,y\in H^+$
	\begin{gather*}
		\Delta(x\prec y)=\Delta(x)\prec \Delta(y),\qquad \Delta(x\cdot y)=\Delta(x)\cdot\Delta(y),\qquad \Delta(x\succ y)=\Delta(x)\succ\Delta(y).	
	\end{gather*}
\end{defi}

This notion of tridendriform bialgebra is such that $(H,\prec,\succ+\cdot,1_H,\Delta,\varepsilon)$ and $(H,\prec+\cdot,\succ,\allowbreak 1_H,\Delta,\varepsilon)$ are both dendriform bialgebras, see Definition 1.3 with $q=1$ from~\cite{Tridend}.
 Let us remind these definitions.

\begin{defi}\label{graded}
	Let $(H,m,1,\Delta,\varepsilon)$ be a bialgebra. We say $H$ is \emph{graded} if there exists a sequence~$(H_n)_{n\in\N}$ of vectorial subspaces of $H$ of finite dimension verifying
	\[
	H=\bigoplus_{n=0}^{+\infty}H_n.
	\]
	Moreover, this grading satisfies
	\begin{itemize}\itemsep=0pt
		\item $\forall i,j\in\N$, $m(H_i\otimes H_j)\subseteq H_{i+j}$,
		\item $\forall n\in\N$, $\Delta(H_n)\subseteq \sum_{i+j=n} H_i\otimes H_j$.
	\end{itemize}
\end{defi}

\begin{defi}\label{connected}
	A graded Hopf algebra $H$ is called \emph{connected} if $H_0$ has dimension $1$.
\end{defi}

Then giving $\mathcal{A}$ the grading $\mathcal{A}_n=\K T_n$ for all $n\in \N$, we also wish that $(\mathcal{A},*,|,\Delta,\varepsilon)$ is a~\emph{graded and connected} bialgebra.
\begin{Not}
	For any tree $t$, $\nf(t)$ is the number of leaves of $t$ and $\deg(t)=\nf(t)-1$.
\end{Not}
 \subsection{Coproduct description} 

 In this section, we will look for a coalgebra structure on $(\mathcal{A},\prec,\cdot,\succ)$.
\begin{defi}\label{defi:candidat}
	We define the following tridendriform algebra morphism:
	\begin{equation*}
		\Delta\colon \ \begin{cases}
			\mathcal{A} \rightarrow \mathcal{A}\overline{\otimes} \mathcal{A} ,\\
			\Y \mapsto \Y\otimes | +|\otimes \Y, \\
			| \mapsto |\otimes |.
		\end{cases}
	\end{equation*}	
\end{defi}
But, as $\mathcal{A}$ is the free tridendriform algebra generated by one element, this morphism exists and is unique because $\mathcal{A}\overline{\otimes} \mathcal{A}$ is also a tridendriform algebra as seen in the previous section. We will check at Theorem~\ref{thm:coassDelta} that $\Delta$ is coassociative.

\begin{Not}
	For any $n\in \N$, we denote by $c_n$ the corolla with $n+1$ leaves
	
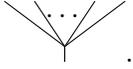
\[
		\begin{tikzpicture}[line cap=round,line join=round,>=triangle 45,x=0.4cm,y=0.4cm]
			\draw (0,0)--(0,0.5);
			\draw (0,0.5)--(-2,2);
			\draw (0,0.5)-- (-1,2);
			\draw (0,1.5) node{$\cdots$};
			\draw (0,0.5)--(2,2);
			\draw (0,0.5)-- (1,2);
		\end{tikzpicture}.
	\]
\end{Not}

\begin{Rq}
	We describe here an explicit technique in order to find one writing of a tree.
 Let $n\in\N$ and $x\in T_n$.
 We write $x$ uniquely as $x=x^{(0)}\vee\cdots \vee x^{(k)}.$ We need to consider four cases:
\begin{itemize}\itemsep=0pt
 \item\textit{Case} 1: $\forall i\in\IEM{0}{k}$, $x^{(i)}=|$.
 Then $x=\Y\cdot c_k$.

 \item\textit{Case} 2: $x^{(0)}\neq |.$
 Then $x=x^{(0)}\succ \big(|\vee x^{(1)}\vee \cdots \vee x^{(k-1)}\vee x^{(k)}\big)$.

 \item\textit{Case} 3: $x^{(k)}\neq |.$
 Then $x=(x^{(0)}\vee \cdots \vee x^{(k-1)}\vee |)\prec x^{(k)}$.

 \item\textit{Case} 4: there exists $i\in \IEM{1}{k-1}$ such that $x^{(i)}\neq |$. Then we obtain
 \begin{align*}
 	x&=x^{(0)}\vee \cdots\vee x^{(i-1)}\vee x^{(i)}\vee x^{(i+1)}\vee \cdots \vee x^{(k)} \\
 	&=\big(x^{(0)}\vee \cdots \vee x^{(i-1)}\vee |\big)\cdot \big(x^{(i)}\vee x^{(i+1)}\vee \cdots \vee x^{(k)}\big) \\
 	&=\big(x^{(0)}\vee \cdots \vee x^{(i)}\big)\cdot \big(| \vee x^{(i+1)}\vee \cdots \vee x^{(k)}\big),
 \end{align*}
 \end{itemize}
where the two writings are equal thanks to the relation $\eqref{tri5}$. Then we apply this process once more on the branches of $x$ and so on. Finally, we get a writing of $x$ over the language $\{\prec,\cdot,\succ\}$ with constant symbol $\big\{\Y\big\}$.
 \label{algo}
 \end{Rq}

 In the following, we will give a lemma that helps us to find an explicit writing of $\Delta$.
\begin{Not}
	For all $x\in \mathcal{A}^+$, we denote $\tilde{\Delta}(x)=\Delta(x)-|{\otimes} \,x-x\, {\otimes} |.$
\end{Not}

 \begin{Lemme}
 	\label{cocalcul}
 	Let $t,s\in \mathcal{A}=\Tridend(\K)$. Then
 	\begin{gather*}
 		\tilde{\Delta}(t\cdot s) =\tilde{\Delta}(t)\cdot \tilde{\Delta}(s)+ (1\otimes t)\cdot \tilde{\Delta}(s)+\tilde{\Delta}(t)\cdot(1\otimes s), \\
 		\tilde{\Delta}(t\prec s) =s\otimes t+ (1\otimes t)\prec \tilde{\Delta}(s)+\tilde{\Delta}(t)\prec (1\otimes s)+\tilde{\Delta}(t)*(s\otimes 1)+\tilde{\Delta}(t)\prec \tilde{\Delta}(s), \\
 		\tilde{\Delta}(t\succ s) =t\otimes s+(1\otimes t)\succ\tilde{\Delta}(s)+(t\otimes 1)*\tilde{\Delta}(s)+\tilde{\Delta}(t)\succ (1\otimes s)+\tilde{\Delta}(t)\succ\tilde{\Delta}(s).
 	\end{gather*}
 \end{Lemme}
\begin{proof}
	Let $\ltimes\in\{ \succ,\cdot,\prec\}$. Let $t,s\in \mathcal{A}$. Then
	\begin{gather*}
		\Delta(t\ltimes s) =\big(1\otimes t+t\otimes 1+\tilde{\Delta}(t)\big)\ltimes \big(1\otimes s+s\otimes 1+\tilde{\Delta}(s)\big) \\
		\hphantom{\Delta(t\ltimes s)}{}=1\otimes t\ltimes s +s\otimes t\ltimes 1+(1\otimes t)\ltimes \tilde{\Delta}(s)+t\otimes 1\ltimes s+ t\ltimes s\otimes 1 \\
		\hphantom{\Delta(t\ltimes s =)}{}+ (t\otimes 1)\ltimes \tilde{\Delta}(s)
		+\tilde{\Delta}(t)\ltimes(a\otimes s)+\tilde{\Delta}(t)\ltimes (s\otimes 1)+\tilde{\Delta}(t)\ltimes \tilde{\Delta}(s).
		\end{gather*}
	We get the lemma applying the definition of $\prec$, $\cdot$ and $\succ$ on the equation above.
	\end{proof}

\begin{defi}
	Let $t$ be a tree, we define the \emph{valence} of a vertex $v$ of $t$ as the number of outgoing edges of $v$.
	A \emph{leaf} of $t$ is an edge of valence $0$. We call a vertex of $t$ \emph{internal} if it is not a leaf.
	
	We say that an edge of $t$ is \emph{internal} if it links two internal vertices.
\end{defi}

\begin{defi}
	Let $t$ be a tree. A \emph{cut} of $t$ is a non-empty choice of internal edges of $t$. We call \emph{empty cut} of $t$, the cut which chooses no edge of $t$ and we call \emph{total cut} the cut ``under the root of $t$''.

	A cut $c$ of $t$ is called \emph{admissible} if all paths from the root to one of the leaves do not meet more than one chosen edge by $c$. The total and empty cuts are admissible cuts.
\end{defi}

\begin{Eg}
On the tree \smash{\raisebox{-0.5\height}{\begin{tikzpicture}[line cap=round,line join=round,>=triangle 45,x=0.25cm,y=0.25cm]
		\draw (0,0)--(0,1);
		\draw (0,1)--(-2,3);
		\draw (0,1)--(2,3);
		\draw (1,2)--(0,3);
		\draw (1.5,2.5)--(1,3);
		\draw (-1.5,2.5)--(-1,3);
		\draw (1,2.3)--(1.75,2.3);
		\draw (0,1.5)--(1.25,1.5);
		\draw (-1.5,2)--(-0.5,2);
\end{tikzpicture}}} the drawn cut is not admissible but \smash{\raisebox{-0.5\height}{\begin{tikzpicture}[line cap=round,line join=round,>=triangle 45,x=0.25cm,y=0.25cm]
		\draw (0,0)--(0,1);
		\draw (0,1)--(-2,3);
		\draw (0,1)--(2,3);
		\draw (1,2)--(0,3);
		\draw (1.5,2.5)--(1,3);
		\draw (-1.5,2.5)--(-1,3);
		\draw (1,2.3)--(1.75,2.3);
		\draw (-1.5,2)--(-0.5,2);
\end{tikzpicture}}} is admissible.
\end{Eg}
\begin{Not}
	Let $t$ be a tree and $c$ be an admissible cut of $t$ which is neither empty, neither total. By withdrawing the chosen edges by $c$, we find some trees that we will denote $G^c_1(t),\dots,G^c_m(t)$, where $m$ is the number of edges chosen by $c$, $G^c_1(t)$ is the left-most tree after withdrawing all edges of $c$, \dots, and $G^c_m(t)$ is the right-most tree.
	
	We also denote by $R^c(t)$ the component after withdrawing the edges of $c$ containing the root of $t$.
	If $c$ is the empty cut, we define $R^c(t):=t$ and $G^c(t):=|.$ If $c$ is the total cut, we define $R^c(t):=|$ and $G^c(t)=t.$
\end{Not}
\begin{Egs}\quad
\begin{itemize}\itemsep=0pt
	\item The empty cut $c$ of the tree $t$ gives $R^c(t)=t.$
	\item The total cut $c$ of $t$ gives $G^c(t)=t.$
	\item Let us consider the tree \raisebox{-0.3\height}{\begin{tikzpicture}[line cap=round,line join=round,>=triangle 45,x=0.2cm,y=0.2cm]
			\draw (0,0)--(0,1);
			\draw (0,1)--(-2,3);
			\draw (0,1)--(2,3);
			\draw (-1.5,2.5)--(-1,3);
			\draw (1.5,2.5)--(1,3);
			\draw (0.5,2)--(1.5,2);
	\end{tikzpicture}}
	with the cut $c$ symbolized by the horizontal line cutting the tree. In this case, $R^c(t)=\balaisg$ and $G^c(t)=\Y.$
	\item With \raisebox{-0.5\height}{\begin{tikzpicture}[line cap=round,line join=round,>=triangle 45,x=0.3cm,y=0.3cm]
		\draw (0,0)--(0,1);
		\draw (0,1)--(0,3);
		\draw (0,1)--(-3,4);
		\draw (0,1)--(3,4);
		\draw (2,3)--(2,4);
		\draw (2,3)--(1,4);
		\draw (0,3)--(0.5,4);
		\draw (0,3)--(-0.5,4);
		\draw (-2,3)--(-1,4);
		\draw (-2.5,3.5)--(-2,4);
		\draw (1,2.5)--(2,2.5);
		\draw (-0.5,2)--(0.5,2);
		\draw (-3,3.35)--(-2,3.35);
	\end{tikzpicture}}, we get $R^c(t)=\raisebox{-0.5\height}{\begin{tikzpicture}[line cap=round,line join=round,>=triangle 45,x=0.15cm,y=0.15cm]
	\draw (0,0)--(0,1);
	\draw (0,1)--(-3,4);
	\draw (0,1)--(0,4);
	\draw (0,1)--(3,4);
	\draw (-2,3)--(-1,4);
\end{tikzpicture}}$ and $G^c_1(t)=\Y$, $G^c_2(t)=\Y$, $G^c_3(t)=\balais.$
	\end{itemize}
\end{Egs}
Observing our results, we get
\begin{Prop}\label{Prop:candidat}
	The map defined at Definition~$\ref{defi:candidat}$ over $\mathcal{A}$ is given for all trees $t$ by the formula
	\begin{gather*}
		\Delta(t)=\sum_{c \text{ \normalfont admissible cut of }t} G^c(t)\otimes R^c(t), \qquad
		\Delta(|)=|\otimes |,
	\end{gather*}
	where $R^c(t)$ is the component of $t$ containing its root and $G^c(t)=G^c_1(t)*\cdots *G^c_k(t)$, where $G^c_i(t)$ are naturally ordered from left to right.
\end{Prop}
\begin{proof}
	We proceed by induction over the number of leaves. As seen before,
	\[
	\Delta\big(\Y\big)=\Y\otimes |+|\otimes \Y.
	\]
	This corresponds to the statement of our proposition.
	
	\textit{Heredity:} let $n\in \N$. Suppose that for all trees with at most $n+1$ leaves, their coproduct is written like in the proposition. Let $t$ be a tree with $n+2$ leaves.
	There exist $k,s\in \N$ such that $k+s=n+3$ and $u\in T_k$, $v\in T_s$ such that one of these equations is true:
	\begin{alignat}{3}
		&t=u\prec v, \qquad&& \text{where} \quad u=u^{(0)}\vee u^{(1)}\vee \cdots \vee u^{(r-1)}\vee |, &\label{cond1} \\
		&t=u\succ v, \qquad && \text{where} \quad v=|\vee v^{(1)}\vee \cdots \vee v^{(r-1)}\vee v^{(r)}, &\label{cond2} \\
		& t=u\cdot v, \qquad && \text{where} \quad u=u^{(0)}\vee \cdots \vee u^{(r-1)}\vee |, & \label{cond3}
	\end{alignat}
using Remark $\ref{algo}$ and $r$ denotes the valence of the root of $t$.
	Under the condition $\eqref{cond1}$, we have
	\begin{gather*}
		 \Delta(t) = \Delta(u)\prec \Delta(v) \\
		\hphantom{\Delta(t)}{}= \bigg(\sum_{c_u \text{ admissible cut of }u} G^{c_u}(u)\otimes R^{c_u}(u)\bigg)\prec \bigg(\sum_{c_v \text{ admissible cut of }v} G^{c_v}(v)\otimes R^{c_v}(v)\bigg) \\
		\hphantom{\Delta(t)}{} = \sum_{\substack{c_u \text{ admissible cut of }u \\
		c_v \text{ admissible cut of }v\\ R^{c_u}\neq |\wedge R^{c_v}\neq | \\
	R^{c_u}\neq u \wedge R^{c_v}\neq v }} G^{c_u}(u)*G^{c_v}(v)\otimes R^{c_u}(u)\prec R^{c_v}(v) +|\otimes u\prec v+u\prec v\otimes |.
	\end{gather*}
	But $u=u^{(0)}\vee u^{(1)}\vee \cdots \vee u^{(r-1)}\vee |$. As a consequence, an admissible cut of $u\prec v=u^{(0)}\vee u^{(1)}\vee \cdots \vee u^{(r-1)}\vee v$ is an admissible cut of $u$ followed by an admissible cut of $v$. Moreover, the branches falling from $v$ fall at the right of the branches from $u$. This implies
	\[
	\Delta(u\prec v)=\Delta(t)=\sum_{c \text{ admissible cut of }t} G^{c}(t)\otimes R^{c}(t).
	\]
	With an analogous idea, we get the same result in the case $\eqref{cond2}$. In the case $\eqref{cond3}$, we get $t=u\cdot v$ with $u=u^{(0)}\vee \cdots \vee u^{(r-1)}\vee |$. As a consequence,
	\begin{align*}
		\Delta(t) &=\bigg( \sum_{c_u \text{ admissible cut of }u} G^{c_u}(u)\otimes R^{c_u}(u)\bigg)\cdot \bigg(\sum_{c_v \text{ admissible cut of }v}G^{c_v}(v)\otimes R^{c_v}(v) \bigg) \\
				&=|\otimes u\cdot v + u\cdot v\otimes |+ \sum_{\substack{c_u \text{ admissible cut of }u \\c_u \text{ admissible cut of }v\\ R^{c_u}\neq |\wedge R^{c_v}\neq | \\
				R^{c_u}\neq u \wedge R^{c_v}\neq v }} G^{c_u}(u)*G^{c_v}(v)\otimes R^{c_u}(u)\cdot R^{c_v}(v).
	\end{align*}
	We denote $v=v^{(0)}\vee \cdots \vee v^{\left(r'\right)}$.
	With the shape of $u$, an admissible cut of $u\cdot v=u^{(0)}\vee \cdots \vee u^{(r-1)}\vee v^{(0)}\vee \cdots \vee v^{\left(r'\right)}$ is an admissible cut of $u$ followed by an admissible cut of $v$. Moreover, the branches of $v$ fall at the right of the branches of $u$.
	This shows that
	\[
	\Delta(t)=\sum_{c \text{ admissible cut of }t} G^{c}(t)\otimes R^{c}(t).
	\tag*{\qed}\]
	\renewcommand{\qed}{}
\end{proof}

\begin{Rq}\label{Rq:multcoprod}
	In the case of multiple generators, the coproduct over $\mathcal{A}(\mathcal{D})$ is the same extending naturally its definition to decorated trees.
\end{Rq}

\begin{thm}\label{thm:coassDelta}
	The coproduct of Proposition~$\ref{Prop:candidat}$ is coassociative and has for counit the linear map defined by
	\[
	\e(t)=\begin{cases}
		1 & \text{\rm if } t=|, \\
		0 &\text{\rm else}.
	\end{cases}
	\]
	As a consequence, $(\mathcal{A},*,|,\Delta,\varepsilon)$ is a connected graded bialgebra and a tridendriform bialgebra. Moreover, $\Delta$ is the unique map making $(\mathcal{A},*,|)$ a tridendriform bialgebra.
\label{coproduit}	
\end{thm}
\begin{proof}
	It is enough to show that $\Delta$ is coassociative, unique and $\varepsilon$ satisfies the counit property.
	First, let us remind that $\mathcal{A}^+$ is the free tridendriform algebra generated by one element. The result of Lemma~\ref{lem:tritens} says that $\mathcal{A}^+\otimes \mathcal{A}^+$ is a tridendriform algebra. Using the universal property of the free tridendriform algebra given by the operad theory, we find the existence of a unique tridendriform algebra morphism such that
	\begin{equation*}
		\Delta \colon \ \begin{cases}
			\mathcal{\mathcal{A}}^+ \rightarrow \mathcal{\mathcal{A}}^+\overline{\otimes} \mathcal{\mathcal{A}}^+ ,\\
			\Y \mapsto |\otimes \Y +\Y\otimes |.
		\end{cases}
	\end{equation*}
	Then we expand the definition of $\Delta$ to $\mathcal{\mathcal{A}}$ by defining $\Delta(|)=|\otimes |.$
	Consequently, we have the uniqueness of $\Delta$. We get two maps
	\begin{gather*}
			({\rm Id} \otimes \Delta)\circ\Delta \colon \ \mathcal{A} \rightarrow \mathcal{A}\otimes (\mathcal{A}\otimes \mathcal{A}), \\
			(\Delta \otimes {\rm Id})\circ\Delta \colon \ \mathcal{A} \rightarrow (\mathcal{A}\otimes \mathcal{A})\otimes \mathcal{A}.
	\end{gather*}
But we have seen that the final spaces of these two maps have the same augmented tridendriform structure in Lemma~\ref{lem:tensassoaugm}. By Lemma~\ref{Lemme:morphtridend}, $(\Delta \otimes {\rm Id})\circ\Delta$ and $({\rm Id} \otimes \Delta)\circ\Delta$ are morphisms of augmented tridendriform algebras. Then we deduce that these two applications are equal over~$\mathcal{A}^+$ because $({\rm Id} \otimes \Delta)\circ\Delta\big(\Y\big)=(\Delta \otimes {\rm Id})\circ\Delta\big(\Y\big)$ and $\Y$ generates $\mathcal{A}^+$ as tridendriform algebra. Moreover, $({\rm Id}\otimes \Delta)\circ\Delta(|)=(\Delta\otimes {\rm Id})\circ \Delta(|)$. So, $(\Delta \otimes {\rm Id})\circ \Delta=({\rm Id}\otimes \Delta)\circ \Delta$ over $\mathcal{A}$. So $\Delta$ is coassociative.
	
	 For the counit property, using Lemma~\ref{Lemme:morphtridend}, we know that $(\varepsilon\otimes {\rm Id})\circ \Delta$ and $({\rm Id} \otimes \varepsilon)\circ \Delta$ are morphisms of augmented tridendriform algebras. Moreover,
	\begin{align*}
		(\varepsilon\otimes {\rm Id})\circ \Delta\big(\Y\big)=\Y=({\rm Id} \otimes \varepsilon)\circ \Delta\big(\Y\big).
	\end{align*}
Thanks to the universal property of the free tridendriform algebra $\mathcal{A}^+$, these two morphisms are equal to ${\rm Id}$. It is clearly the same for $\langle | \rangle$. So, it is the identity over $\mathcal{A}$.
\end{proof}

\begin{defi}
	For any bialgebra $(H,*,1,\Delta,\varepsilon)$, we define
	\[
	\Prim(H):=\big\{ x\in H \mid \tilde{\Delta}(x)=0 \big\}.
	\]
\end{defi}
\begin{Cor}
	Every tree $t$ element of $\Prim(\mathcal{A})$ is a corolla.
\end{Cor}
\begin{proof}
	Let $t\in\Prim(\mathcal{A})$. If $t$ has two leaves, then $t=\Y$.
	Let us suppose $t$ has at least three leaves.
	Then using the fact that \Y generates the tridendriform algebra $\mathcal{A}$, we deduce the existence of two trees $u,v$ which have strictly less leaves than $t$ such that
	\[
	t=u\prec v \qquad \text{or} \qquad t=u\cdot v \qquad \text{or} \qquad t=u\succ v.
	\]
	By Lemma~\ref{cocalcul}, the only possibility for $t$ to be primitive is $t=u\cdot v$ with $u,v\in\Prim(\mathcal{A})$. By induction hypothesis, $u$ and $v$ are corollas. So $t$ is a corolla.
\end{proof}

\begin{Egs}\quad
	\begin{itemize}\itemsep=0pt
		\item $\Y \cdot \Y=|\vee (|*|)\vee |=\balais$ is primitive.
		\item $\balais\cdot \Y=|\vee |\vee (|*|)\vee |= \raisebox{-0.3\height}{\begin{tikzpicture}[line cap=round,line join=round,>=triangle 45,x=0.2cm,y=0.2cm]
				\draw (0,0)-- (0,1);
				\draw (0,1)-- (1.5,2);
				\draw (0,1)-- (0.5,2);
				\draw (0,1)-- (-0.5,2);
				\draw (0,1)-- (-1.5,2);
		\end{tikzpicture}}$ is primitive.
	\end{itemize}
\end{Egs}

\subsection{Description of the product}
\subsubsection{Quasi-shuffles}
In this part, our objective is to give non inductive descriptions of the product $*$ and the three others operations $\prec$, $\cdot$, $\succ$. We shall need the following definition:
\begin{defi}[quasi-shuffle]
	Let $k,l\in\N\setminus \{0\}$. A \emph{$(k,l)$-quasi-shuffle} is a surjective map $\sigma\colon\IEM{1}{k+l}\twoheadrightarrow\IEM{1}{n}$ surjective for some positive integer $n$ such that
	\[
	\sigma(1)<\cdots<\sigma(k) \qquad \text{and} \qquad \sigma(k+1)<\cdots<\sigma(k+l).
	\]
	We will denote $\batc(k,l)$ the set of all $(k,l)$-quasi-shuffles.
\end{defi}

The reader may find more details about quasi-shuffles and tridendriform algebras in~\cite{Bruhat}.

\begin{Rq}
	We will use the following notation to describe quasi-shuffles. Let $k,l\in\N\setminus\{0\}$. Let $\sigma$ be a surjective map of $\IEM{1}{k+l}$ into some $\IEM{1}{n}$. We will denote $\sigma$ like this
	\[
	(\sigma(1),\dots,\sigma(k),\sigma(k+1),\dots,\sigma(k+l)).
	\]
\end{Rq}
\begin{Rq}
	Let $\sigma\in\batc(k,l)$, then $\sigma^{-1}(\{1\})$ is either $\{1\}$, $\{k+1\}$ or $\{1,k+1\}$.
\end{Rq}

\subsubsection{Comb representation of a tree}

Let $t$ be a tree. It $t$ can always be seen as
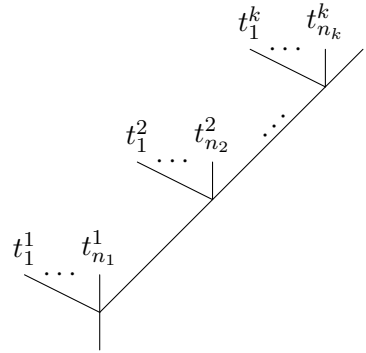
\begin{center}
\raisebox{-0.5\height}{\begin{tikzpicture}[line cap=round,line join=round,>=triangle 45,x=0.5cm,y=0.5cm]
\draw (0,0)--(0,1);
\draw (0,1)--(-2,2) node[above]{$t^1_1$};
\draw (0,1)--(0,2) node[above]{$t^1_{n_1}$};
\draw (-1,2) node{$\cdots$};
\draw (0,1)--(7,8);
\draw (3,4)--(1,5) node[above]{$t^2_1$};
\draw (3,4)--(3,5) node[above]{$t^2_{n_2}$};
\draw (2,5) node{$\cdots$};
\draw (4.7,6) node[rotate=45]{$\cdots$};
\draw (6,7)--(4,8) node[above]{$t^k_1$};
\draw (6,7)--(6,8) node[above]{$t^k_{n_k}$};
\draw (5,8) node {$\cdots$};
\end{tikzpicture}}
\end{center}
where $k$ is the number of nodes on the right-most branch of $t$ and for $i\in\IEM{1}{k}$, $n_i+1$ is the number of sons of the $i$-th node of this branch.

\begin{Not}
	Let $F=t_1\dots t_n$ be a forest composed of $n$ trees. Instead of writing
	\begin{center}
		\raisebox{-0.5\height}
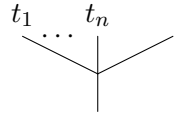
{\begin{tikzpicture}[line cap=round,line join=round,>=triangle 45,x=0.5cm,y=0.5cm]
				\draw (0,0)--(0,1);
				\draw (0,1)--(-2,2) node[left,above]{$t_1$};
				\draw (0,1)--(0,2) node[right,above]{$t_{n}$};
				\draw (-1,2) node{$\cdots$};
				\draw (0,1)--(2,2);
		\end{tikzpicture}}
	\end{center}
	we will write
	\raisebox{-0.3\height}{\begin{tikzpicture}[line cap=round,line join=round,>=triangle 45,x=0.3cm,y=0.3cm]
			\draw (0,0)--(0,1);
			\draw (0,1)--(-1,2) node[left]{$F$};
			\draw (0,1)--(1,2);
	\end{tikzpicture}}.
\end{Not}

As a consequence, we deduce that all tree $t$ can be seen as
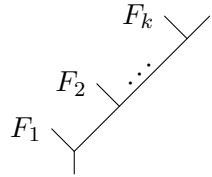
\begin{center}
	\raisebox{-0.5\height}{\begin{tikzpicture}[line cap=round,line join=round,>=triangle 45,x=0.3cm,y=0.3cm]
		\draw (0,0)--(0,1);
		\draw (0,1)--(-1,2) node[left]{$F_1$};
		\draw (0,1)--(6,7);
		\draw (2,3)--(1,4) node[left]{$F_2$};
		\draw (3.,4.5) node[rotate=45]{$\cdots$};
		\draw (5,6)--(4,7) node[left]{$F_k$};
	\end{tikzpicture}}
\end{center}
where $F_1,\dots,F_k$ are non-empty forests.
\begin{defi}
	The writing defined above is the writing of $t$ as a \emph{right comb}.
	Proceeding the same way, looking at the left branch of the root of $t$, we get the writing of $t$ as a \emph{left comb}.
\end{defi}

\subsubsection{Action of the quasi-shuffles over ordered pairs of trees}

Let $t$, $s$ two trees different from $|$. We look $t$ as a right comb and $s$ as a left comb. In other words, we put
	\begin{align*}
		&t=\peignedroit{1}{2}{k}
		\qquad \text{and} \qquad s=\peignegauche{k+1}{k+2}{k+l}
	\end{align*}
where for all $i\in\IEM{1}{k+l}$, $F_i$ is a non-empty forest. Here, $k$ represents the number of nodes on the right-most branch of $t$ and $l$ is the number of nodes on the left-most branch of $s$.

\begin{defi}\label{def:quasiaction}
	Let $t$ and $s$ be two trees which respective right comb representation and left comb representation are given above.
	We denote by $k$ (respectively $l$) the number of nodes on the right-most (respectively left-most) branch of $t$ (respectively $s$). Let $\sigma$ be a $(k,l)$-quasi-shuffle which has for image $\IEM{1}{n}$ for $n$ a positive integer. We denote $\sigma(t,s)$ the tree obtained this way:
	\begin{enumerate}\itemsep=0pt
		\item We first consider the ladder with $n$ nodes:
		\[
			\raisebox{-0.5\height}
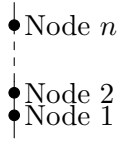
{\begin{tikzpicture}[line cap=round,line join=round,>=triangle 45,x=0.3cm,y=0.3cm]
				\begin{scope}
					\draw (0,0)--(0,2.5);
					\draw[dashed] (0,2.5)--(0,5);
					\draw (0,5)--(0,6);
					\filldraw [black] (0,1) circle (2pt) node[anchor=west]{Node $1$};
					\filldraw [black] (0,2) circle (2pt) node[anchor=west]{Node $2$};
					\filldraw [black] (0,5) circle (2pt) node[anchor=west]{Node $n$};
				\end{scope}
			\end{tikzpicture}}
		\]
		\item For all $i\in\IEM{1}{k},$ we graft $F_i$ as the \emph{left} son at the node $\sigma(i)$.
		\item For all $i\in\IEM{k+1}{k+l},$ we graft $F_i$ as \emph{right} son at the node $\sigma(i)$.
	\end{enumerate}
\end{defi}

\begin{Eg}\label{Eg:quasiaction}
	Consider the $(2,2)$-quasi-shuffle $\sigma=(1,3,2,3)$. Let us take
\[
t= \raisebox{-0.3\height}{\begin{tikzpicture}[line cap=round,line join=round,>=triangle 45,x=0.2cm,y=0.2cm]
			\draw (0,0)--(0,1);
			\draw (0,1)--(-1,2) node[left]{$F_1$};
			\draw (0,1)--(2,3);
			\draw (1,2)--(0,3) node[left,above]{$F_2$};
	\end{tikzpicture}} \qquad \text{and}\qquad s= \raisebox{-0.3\height}{\begin{tikzpicture}[line cap=round,line join=round,>=triangle 45,x=0.2cm,y=0.2cm]
			\draw (0,0)--(0,1);
			\draw (0,1)--(1,2) node[right]{$F_{3}$};
			\draw (0,1)--(-2,3);
			\draw (-1,2)--(0,3) node[right,above]{$F_{4}$};
	\end{tikzpicture}},\qquad \text{then}\qquad \sigma(t,s)=\raisebox{-0.5\height}{\begin{tikzpicture}[line cap=round,line join=round,>=triangle 45,x=0.3cm,y=0.3cm]
			\draw (0,0)--(0,4);
			\draw (0,1)--(-1,2) node[left]{$F_1$};
			\draw (0,2)--(1,3) node[right]{$F_3$};
			\draw (0,3)--(-1,4) node[left]{$F_2$};
			\draw (0,3)--(1,4) node[right]{$F_{4_{\vphantom{\big|}}}$};
	\end{tikzpicture}}.
\]
\end{Eg}

\begin{Rq}\label{Rq:multprod}
		For multiple generators, we extend the definition of this product naturally over~$\mathcal{A}(X)$ except when a node of this ladder has left and right sons. In this case, the decoration of this node is the concatenation of the decorations of the roots of the forests grafted at this node.	
		\begin{Eg}
			Let $a,b\in\mathcal{D}$, then
			\[
			(1,1)\big(\Ydec{a}, \Ydec{b}\big)={\raisebox{-0.3\height}{\begin{tikzpicture}[line cap=round,line join=round,>=triangle 45,x=0.3cm,y=0.3cm]
						\draw (0,0)--(0,2);
						\draw (0,1)-- (-1,2);
						\draw (0,1)-- (1,2);
						\draw (0,1) node[left] {\footnotesize{$ab$}};
				\end{tikzpicture}}}.
			\]
		\end{Eg}
\end{Rq}

\begin{Not}
	 Given $F=t_1\dots t_n$ a forest, we define
	\begin{gather*}
		\nf(F)=\sum_{i=1}^{n}\nf(t_i), \qquad \deg(F)=\sum_{i=1}^{n}\deg(t_i).
	\end{gather*}
\end{Not}
\begin{Rq}\label{battage}
	For all trees $t$, $s$ of $\mathcal{A}$ and for all $\sigma,\gamma\in\batc(k,l)$, we have
	\begin{equation*}
		\sigma(t,s)= \gamma(t,s) \iff \sigma= \gamma.
	\end{equation*}
\end{Rq}
\begin{proof}
	The ladder used to build $\sigma(t,s)$ and $\gamma(t,s)$ is a tree along the path going from the root to the $\nf(t)$-th leaf. Comparing right and left sons of each node of the ladder we see $\sigma= \gamma.$
\end{proof}

\subsubsection{Description of the product}

Let $t$, $s$ be two trees different from $|$. We will suppose that our trees are written as follows:
\begin{align*}
	&t=\peignedroit{1}{2}{k} \qquad \text{and} \qquad s=	\peignegauche{k+1}{k+2}{k+l}
\end{align*}
\begin{Not}
	Let $F=t_1\dots t_k$ be a forest and $t$ a tree. We denote by $F\vee t$ the following tree:
	\[
	t_1\vee \dots \vee t_k\vee t.
	\]
\end{Not}

\begin{thm}\label{produit}
	Let $t$, $s$ be two trees different from $|$ as described above. Then
	\[
	t*s=\sum_{\sigma\in \batc(k,l)}\sigma(t,s).
	\]
\end{thm}
\begin{proof}
	We prove this theorem by induction on the sum of $k$ and $l$. As $t,s\neq |$, $k,l\in\N\setminus\{0\}$.
	To initialize the induction, the only case to check is $k+l=2$, that is to say $k=1$ and $l=1$. Let us remind that $T_1=\big\lbrace\Y \big\rbrace$ and $\batc(1,1)=\{(1,1),(1,2),(2,1)\}$. Then
	\begin{align*}
		&(1,1)\big(\Y,\Y\big)+(2,1)\big(\Y,\Y\big)+(1,2)\big(\Y,\Y\big)\\
&\qquad=\raisebox{-0.3\height}{\begin{tikzpicture}[line cap=round,line join=round,>=triangle 45,x=0.2cm,y=0.2cm]
				\draw (0,0)--(0,2);
				\draw (0,1)--(1,2);
				\draw (0,1)--(-1,2);
		\end{tikzpicture}} +
		\raisebox{-0.3\height}{\begin{tikzpicture}[line cap=round,line join=round,>=triangle 45,x=0.2cm,y=0.2cm]
				\draw (0,0)--(0,2);
				\draw (0,0.66)--(1,1.66);
				\draw (0,1.33)--(-1,2);
		\end{tikzpicture}}+
		\raisebox{-0.3\height}{\begin{tikzpicture}[line cap=round,line join=round,>=triangle 45,x=0.2cm,y=0.2cm]
				\draw (0,0)--(0,2);
				\draw (0,0.66)--(-1,1.66);
				\draw (0,1.33)--(1,2);
		\end{tikzpicture}}
		=\raisebox{-0.3\height}{\begin{tikzpicture}[line cap=round,line join=round,>=triangle 45,x=0.2cm,y=0.2cm]
				\draw (0,0)--(0,2);
				\draw (0,1)--(1,2);
				\draw (0,1)--(-1,2);
		\end{tikzpicture}}+
		\raisebox{-0.3\height}{\begin{tikzpicture}[line cap=round,line join=round,>=triangle 45,x=0.2cm,y=0.2cm]
				\draw (0,0)--(0,1);
				\draw (0,1)--(-1,2);
				\draw (0,1)--(1,2);
				\draw (-0.5,1.5)--(0,2);
		\end{tikzpicture}} +
		\raisebox{-0.3\height}{\begin{tikzpicture}[line cap=round,line join=round,>=triangle 45,x=0.2cm,y=0.2cm]
				\draw (0,0)--(0,1);
				\draw (0,1)--(-1,2);
				\draw (0,1)--(1,2);
				\draw (0.5,1.5)--(0,2);
		\end{tikzpicture}}
	=\Y*\Y.
	\end{align*}
	Secondly, we check the heredity. Let $(k,l)\in\N\setminus\{0\}$ such that $k+l\geq 2$. Let us suppose that for all couples $(k',l')\in\N\setminus\{0\}$ such that $k'+l'<k+l$, we have the theorem. Then
	\allowdisplaybreaks[4]
	\begin{align*}
t*s={}&t\prec s+t\cdot s+t\succ s \\
={}&F_1 \vee \left( \peignedroit{2}{3}{k} * \peignegauche{k+1}{k+2}{k+l} \right) \\
&{}+\left( \peignedroit{1}{2}{k}*\peignegauche{k+2}{k+3}{k+l} \right)\vee F_{k+1} \\
&{}+F_1\vee \left( \peignedroit{2}{3}{k}* \peignegauche{k+2}{k+3}{k+l} \right)\vee F_{k+1}		\\
={}&F_1\vee \left(\sum_{\tau\in\batc(k-1,l)} \tau\left(\peignedroit{2}{3}{k},\peignegauche{k+1}{k+2}{k+l}\right)\right) \\
&{}+F_1\vee \left( \sum_{\gamma\in \batc(k-1,l-1)} \gamma\left( \peignedroit{2}{3}{k},\peignegauche{k+2}{k+3}{k+l} \right)\right)\vee F_{k+1} \\
&{}+\left( \sum_{\delta\in \batc(k,l-1)}\delta\left(\peignedroit{1}{2}{k},\peignegauche{k+2}{k+3}{k+l}\right) \right)\vee F_{k+1} \\
={}&\sum_{\substack{\tau\in \batc(k-1,l) \\ \sigma_1=(1,\tau(1)+1,\dots,\tau(k+l-1)+1)}}\sigma_1(t,s) \\
&{}+\sum_{\substack{\gamma\in \batc(k-1,l-1) \\ \sigma_2=(1,\gamma(1)+1,\dots,\gamma(k-1)+1,1,\gamma(k)+1,\dots,\gamma(k+l-2)+1)}}\sigma_2(t,s) \\
&{}+\sum_{\substack{\delta\in \batc(k,l-1) \\ \sigma_3=(\delta(1)+1,\dots,\delta(k)+1,1,\delta(k+1)+1,\dots,\delta(k+l-1)+1)}}\sigma_3(t,s).
\end{align*}
Therefore,
\begin{align*}
t*s={}&\sum_{\substack{\sigma_1\in\batc(k,l) \\ \sigma_1^{-1}(\{1\})=\{1\}}}\sigma_1(t,s) +\sum_{\substack{\sigma_2\in\batc(k,l) \\ \sigma_2^{-1}(\{1\})=\{1,k+1\}}}\sigma_2(t,s)+ \sum_{\substack{\sigma_1\in\batc(k,l) \\ \sigma_3^{-1}(\{1\})=\{k+1\}}}\sigma_3(t,s)\\
={}& \sum_{\sigma\in \batc(k,l)} \sigma(t,s). \tag*{\qed}
	\end{align*}
\renewcommand{\qed}{}
\allowdisplaybreaks[1]
\end{proof}

Thanks to the previous theorem, we get
\begin{Cor} \label{Cor:3prodescription}
	Let $t$, $s$ be two trees different from $|$. Then
	\begin{gather*}
		t\prec s=\sum_{\substack{\sigma\in\batc(k,l) \\ \sigma^{-1}(1)=\{1\}}} \sigma(t,s),
		\qquad t\cdot s=\sum_{\substack{\sigma\in\batc(k,l) \\ \sigma^{-1}(1)=\{1,k+1\}}} \sigma(t,s),
		\qquad t\succ s=\sum_{\substack{\sigma\in\batc(k,l) \\ \sigma^{-1}(1)=\{k+1\}}} \sigma(t,s).
	\end{gather*}
\end{Cor}

\begin{proof}
	We just refer to the previous proof.
\end{proof}

\begin{Rq}
	In~\cite{Tridend+co}, the authors describe the tridendriform structure over Schr\"oder trees with \emph{stuffle paths} in Section 7.1.2. A stuffle path is a sequence $\left(\sigma^{-1}(\{i\})\right)_{i\in \IEM{1}{n}}$ where $\sigma$ is a $(k,l)$-quasi-shuffle and $n=\max(\sigma).$ As the set of quasi-shuffles is in bijection with stuffle paths, it gives an equivalent description of the tridendriform structure. But it is still defined inductively.
\end{Rq}
\begin{Rq}
		For many generators, we extend these definitions to decorated trees thanks to Lemma~\ref{Rq:multprod}.
\end{Rq}
Some quasi-shuffle algebras are studied in~\cite{Quasi-shuffle} nearby Proposition $4$. In the case of quasi-shuffle algebras over words, the tridendriform structure is given by shuffling letters. Otherwise, we can also define a tridendriform structure looking at from which word comes the last letter (and not the first). If we do so with Schr\"oder trees, the relation~\eqref{tri1} is not satisfied.

\subsection[(3,1)-dendriform bialgebras quotients]{$\boldsymbol{(3,1)}$-dendriform bialgebras quotients}
\begin{defi}
	Let $(H,\prec,\cdot,\succ,1,\Delta,\varepsilon)$ be $(3,1)$-dendriform bialgebra. We say that $I$ is a~$(3,1)$-dendriform biideal if $I$ is a tridendriform ideal and $I$ is a coideal of $H$.
\end{defi}
With these definitions, it follows:
\begin{Prop}
	Let $(H,\prec,\cdot,\succ,1,\Delta,\varepsilon)$ be a $(3,1)$-dendriform bialgebra. Let $I$ be a $(3,1)$-dendriform biideal. Then $\faktor{H}{I}$ with the quotient structure is a $(3,1)$-dendriform bialgebra.
\end{Prop}

\section[Graded dual of A]{Graded dual of $\boldsymbol{\mathcal{A}}$} \label{sec4}
\subsection{Tridendriform coalgebra}
We consider $(\mathcal{A},*,|,\Delta,\varepsilon)$ as a graded and connected Hopf algebra (see Definitions~\ref{graded} and~\ref{connected}).
We will now describe the graded dual of $\mathcal{A}$, denoted by $\mathcal{A}^{\circledast}$. Let us consider the bilinear map
\begin{gather*}
	\langle \cdot,\cdot \rangle \colon \
		\mathcal{A}\times \mathcal{A} \rightarrow \K,
\end{gather*}
 defined by $\left\langle t,s \right\rangle=\delta_{t,s}$ for $t$, $s$ two trees. Through this pairing, we identify $\mathcal{A}^{\circledast}$ with $\mathcal{A}$ as vector spaces with
\begin{equation}\label{eq:pairing}
	\Phi\colon \ \begin{cases}
		\mathcal{A} \rightarrow \mathcal{A}^{\circledast}, \\
		a \mapsto \begin{cases}
			\mathcal{A} \rightarrow \K, \\
			b \mapsto \langle a,b \rangle.
		\end{cases}
	\end{cases}
\end{equation}
Let us put for all $\ltimes\in\{\prec,\cdot,\succ,*\}$, $\Delta_{\ltimes}={\ltimes}^{\circledast}.$
In order to be readable, we will write $\Delta_{\bullet}$ instead of $\Delta_{\cdot}$. Moreover, we will write $\Delta_*=\Delta$. Then in ${\mathcal{A}^{\circledast}}^+$, we have
\[
\Delta=\Delta_{\prec}+\Delta_{\bullet}+\Delta_{\succ}.
\]
Moreover, for all $a\otimes b\in \mathcal{A}\overline{\otimes} \mathcal{A}$ and $t\in \mathcal{A}\simeq \mathcal{A}^{\circledast}$:
\[
 \langle \Delta_{\ltimes}(t),a\otimes b \rangle= \langle t,a\ltimes b \rangle.
\]
By duality, choosing all the quadruples $(\ltimes,\rtimes,\ltimes',\rtimes')$ elements of~\eqref{quadru} and using the relations~\eqref{tri1}--\eqref{tri7}, we get in ${\mathcal{A}^{\circledast}}^+$
\begin{gather}
	({\rm Id} \otimes \Delta )\circ \Delta_{\prec} =(\Delta_{\prec}\otimes {\rm Id})\circ \Delta_{\prec}, \label{cotri1}\\
	({\rm Id} \otimes \Delta_{\prec} )\circ \Delta_{\succ} =(\Delta_{\succ}\otimes {\rm Id})\circ \Delta_{\prec}, \label{cotri2} \\
	({\rm Id} \otimes \Delta_{\succ} )\circ \Delta_{\succ} =(\Delta\otimes {\rm Id})\circ \Delta_{\succ}, \label{cotri3} \\
	({\rm Id} \otimes \Delta_{\bullet} )\circ \Delta_{\succ} =(\Delta_{\succ}\otimes {\rm Id})\circ \Delta_{\bullet}, \label{cotri4}\\
	({\rm Id} \otimes \Delta_{\succ} )\circ \Delta_{\bullet} =(\Delta_{\prec}\otimes {\rm Id})\circ \Delta_{\bullet}, \label{cotri5}\\
	({\rm Id} \otimes \Delta_{\prec} )\circ \Delta_{\bullet} =(\Delta_{\bullet}\otimes {\rm Id})\circ \Delta_{\prec}, \label{cotri6}\\
	({\rm Id} \otimes \Delta_{\bullet} )\circ \Delta_{\bullet} =(\Delta_{\bullet}\otimes {\rm Id})\circ \Delta_{\bullet}. \label{cotri7}
\end{gather}

\begin{Rq}
Each map described above are maps from ${\mathcal{A}^{\circledast}}^+$ into a space where each term is well defined according to the Remark~\ref{Rq:relations+}.
\end{Rq}

\begin{Eg}
	We have
	\begin{gather*}
		({\rm Id}\otimes \Delta)\circ \Delta_{\prec}\big(\Y\big)=({\rm Id}\otimes \Delta)\big(\Y\otimes |\big)=\Y\otimes |\otimes |, \\
		(\Delta_{\prec}\otimes {\rm Id})\circ \Delta_{\prec}\big(\Y\big)=(\Delta_{\prec}\otimes {\rm Id})\big(\Y\otimes |\big)=\Y\otimes |\otimes |.
	\end{gather*}
\end{Eg}

Then we can break our coproduct $\Delta$ over ${\mathcal{A}^{\circledast}}^+$ in three pieces such that the relations~\eqref{cotri1}--\eqref{cotri7} are verified. In the following, we define as in~\cite[Definition 5.3]{Tridend+co}:
\begin{defi}
	A \emph{tridendriform coalgebra} is a coalgebra $\big(C,\tilde{\Delta}\big)$, where there exist three maps
	\begin{gather*}
		\tilde{\Delta}_{\prec}\colon \ C\rightarrow C\otimes C, \qquad \Delta_{\bullet}\colon \ C\rightarrow C\otimes C, \qquad \tilde{\Delta}_{\succ}\colon \ C\rightarrow C\otimes C,
	\end{gather*}
	such that the relations \eqref{cotri1}--\eqref{cotri7} are true for all $(x,y,z)\in C^3$ replacing each $\Delta_{\ltimes}$ by $\tilde{\Delta}_{\ltimes}$ for $\ltimes\in\{\prec,\cdot,\succ,*\}$. We also have $\tilde{\Delta}=\tilde{\Delta}_{\prec}+\Delta_{\bullet}+\tilde{\Delta}_{\succ}.$
\end{defi}

\begin{defi}
	Let $\bigl(C,\tilde{\Delta}_{\prec},\Delta_{\bullet},\tilde{\Delta}_{\succ}\bigr)$ be a tridendriform coalgebra.
	The associated \emph{augmented tridendriform coalgebra} is the vector space $\overline{C}=1_C\cdot\K\oplus C$ where we expand our coproduct $\tilde{\Delta}$ as a coproduct $\Delta$ on $1_C\cdot\K$ by
	\[
	\Delta(1_C)=1_C\otimes 1_C.
	\]
	We also define for all $x\in C$
	\begin{gather*}
		\Delta(x)=\tilde{\Delta}(x)+1_C\otimes x+x\otimes 1_C, \\
		\Delta_{\prec}(x)=\tilde{\Delta}_{\prec}(x)+x\otimes 1_C, \\
		\Delta_{\succ}(x)=\tilde{\Delta}_{\succ}(x)+1_C\otimes x.
	\end{gather*}	
As a consequence, $\Delta=\Delta_{\prec}+\Delta_{\bullet}+\Delta_{\succ}.$
Note that we do not define $\Delta_{\prec}(1_C)$, neither $\Delta_{\bullet}(1_C)$, nor~$\Delta_{\succ}(1_C)$. As a consequence of these definitions, the relations \eqref{cotri1}--\eqref{cotri7} are satisfied over $C$.
\end{defi}

\subsection[(1,3)-dendriform bialgebras]{$\boldsymbol{(1,3)}$-dendriform bialgebras}

We will study the compatibilities between the product $m$ of $\mathcal{A}^{\circledast}$ and the coproducts $\Delta_{\prec}$, $\Delta_{\bullet}$ and~$\Delta_{\succ}$ of $\mathcal{A}^{\circledast}$. We already know that $\Delta \circ m=m_{\mathcal{A}^{\circledast}\otimes \mathcal{A}^{\circledast}}\circ (\Delta \otimes \Delta)$. Let $f,g\in{\mathcal{A}^{\circledast}}^+$ and $u,v\in \mathcal{A}^+$. Then $f(|)=g(|)=0$. Let $\ltimes\in\{\prec,\cdot,\succ\}$.
Thanks to our construction of $\Delta$ preserving the tridendriform structure, we get
\begin{align*}
	& \langle \Delta_{\ltimes}\circ m(f\otimes g),u\otimes v \rangle= \langle m(f\otimes g), u\ltimes v \rangle
	= \langle f\otimes g,\Delta(u\ltimes v) \rangle
	= \langle f\otimes g,\Delta(u)\ltimes \Delta(v) \rangle.
\end{align*}
For $x\in \mathcal{A}$, we also denote $\tilde{\Delta}(x)=x'\otimes x'',$ omitting the summation symbol to shorten the writing.
Therefore, the results of the Lemma~\ref{cocalcul} can be written as
\begin{gather*}
\tilde{\Delta}(u\prec v)=v\otimes u+v'\otimes u\prec v''+u'*v\otimes u''+u'\otimes u''\prec v+u'*v'\otimes u''\prec v'', \\
\tilde{\Delta}(u\succ v)=u\otimes v+u'*v'\otimes u''\succ v'' +u'\otimes u''\succ v+u*v'\otimes v''+v'\otimes u\succ v'',\\
\tilde{\Delta}(u\cdot v)=u'*v'\otimes u''\cdot v'' +u'\otimes u''\cdot v + v'\otimes u\cdot v''.
\end{gather*}

\begin{Not}
	 For all $\ltimes\in\{\prec,\cdot,\succ\}$ and for all $x\in \mathcal{A}^{\circledast}$, we denote
\begin{align*}
	\Delta_{\ltimes}(x)=x'_{\ltimes}\otimes x''_{\ltimes}
\end{align*}
omitting the sum to shorten the writing. This is the \emph{generalized Sweedler's notation}.
\end{Not}
\begin{Rq}
	For all $f\in \mathcal{A}^{\circledast}$ and for all $u,v\in \mathcal{A}$,
	\[
	 \langle f,u\ltimes v \rangle=f(u\ltimes v)=\Delta_{\ltimes}(f)(u\otimes v)=f'_{\ltimes}(u)f''_{\ltimes}(v).
	\]
\end{Rq}

From those equations, we get for any $f,g\in {\mathcal{A}^{\circledast}}^+$
\begin{gather}
	\Delta_{\prec}(fg)=g\otimes f+ g'_{\prec}\otimes fg''_{\prec}+f'_{\prec}g\otimes f''_{\prec}+f'g'_{\prec}\otimes f''g''_{\prec}, \label{compa1} \\
	\Delta_{\bullet}(fg)=f'g'_{\bullet}\otimes f''g''_{\bullet}+fg'_{\bullet}\otimes g''_{\bullet} + g''_{\bullet}\otimes fg'_{\bullet}, \label{compa2} \\
	\Delta_{\succ}(fg)=f\otimes g +f'g'_{\succ}\otimes f''g''_{\succ}+f'\otimes f''g+g'_{\succ}\otimes fg''_{\succ}. \label{compa3}	
\end{gather}
We detail how to obtain $\eqref{compa1}$. The other verifications are left to the reader,
\begin{gather*}
	\big\langle \Delta_{\prec}(fg),u\otimes v \big\rangle = \big\langle f\otimes g,\tilde\Delta(u\prec v) \big\rangle \\
	\qquad=\big\langle f\otimes g,v\otimes u+v'\otimes u\prec v''+u'*v\otimes u''+u'\otimes u''\prec v+u'*v'\otimes u''\prec v'' \big\rangle \\
	\qquad=\langle g\otimes f,u\otimes v \rangle+\big\langle f\otimes g'_{\prec}\otimes g''_{\prec}, v'\otimes u\otimes v'' \big\rangle+ \big\langle f'_{\prec}\otimes f''_{\prec}\otimes g,u'\otimes v\otimes u'' \big\rangle\\
	\qquad\quad+\big\langle f\otimes g'_{\prec} \otimes g''_{\prec},u'\otimes u''\otimes v \big\rangle+\big\langle f'\otimes f''\otimes g'_{\prec}\otimes g''_{\prec},u'\otimes v'\otimes u'' \otimes v'' \big\rangle \\
	\qquad= \big\langle g\otimes f+ g'_{\prec}\otimes fg''_{\prec}+f'_{\prec}g\otimes f''_{\prec}+f'g'_{\prec}\otimes f''g''_{\prec},u\otimes v \big\rangle.
\end{gather*}
So we get $\eqref{compa1}.$
This encourages us to establish the following definition:
\begin{defi}
	A \emph{$(1,3)$-dendriform bialgebra} is a family $(H,m,\eta,\Delta_{\prec},\Delta_{\bullet},\Delta_{\succ},\varepsilon)$
	such that
	\begin{itemize}\itemsep=0pt
		\item $(H,m,\eta)$ is a unitary algebra.
		\item $(H,\Delta_{\prec},\Delta_{\bullet},\Delta_{\succ})$ is an augmented tridendriform coalgebra and $\varepsilon$ is the counit of $(H,\Delta)$ where $\Delta=\Delta_{\prec}+\Delta_{\bullet}+\Delta_{\succ}.$
		\item The relations \eqref{compa1}, \eqref{compa2} and \eqref{compa3} are satisfied for all $f,g\in H.$
	\end{itemize}
\end{defi}
\begin{Eg}
	The bialgebra $(\mathcal{A}^{\circledast},m,\eta,\Delta,\varepsilon)$ with three maps $\Delta_{\prec}$, $\Delta_{\bullet}$ and $\Delta_{\succ}$ defined above is a $(1,3)$-dendriform bialgebra.
\end{Eg}
\begin{Rq}
	If $H$ is a graded $(3,1)$-dendriform bialgebra, then its graded dual $H^{\circledast}$ is a $(1,3)$-dendriform bialgebra. Moreover, if $H$ is a graded $(1,3)$-dendriform bialgebra, $H^{\circledast}$ is a $(3,1)$-dendriform bialgebra.
\end{Rq}

\subsection[Study of the coproduct of A\^{}\{circledast\}]{Study of the coproduct of $\boldsymbol{\mathcal{A}^{\circledast}}$}

\begin{defi}[lightning decomposition of a tree]
	Let $t$ be a tree and $i\in\IEM{1}{\nf(t)}$. We call the \emph{lightning decomposition} of $t$ on the leaf number $i$, the couple of trees obtained by hitting a~lightning on $t$ at the leaf number $i$ going along the path from that leaf to the root. This lightning splits the tree into two pieces and by contracting unnecessary nodes, we obtain an ordered pair of two trees $(t_1, t_2).$
\end{defi}
\begin{Eg}
\[
\raisebox{-0.3\height}{\begin{tikzpicture}[line cap=round,line join=round,>=triangle 45,x=0.3cm,y=0.3cm]
			\draw (0,0)--(0,1);
			\draw (0,1)--(-3,4);
			\draw (0,1)--(3,4);
			\draw (2,3)--(1,4);
			\draw (-2,3)--(-2,4);
			\draw (-2,3)--(-1,4);
			\draw (-2,3)--(-3,4);
			\draw (-1.9,4) node[above]{\textcolor{orange}{$\lightning$}};
			\draw[thick,orange] (-2,4)--(-2,3)--(0,1)--(0,0);
	\end{tikzpicture}} \rightarrow{\raisebox{-0.4\height}{\begin{tikzpicture}[line cap=round,line join=round,>=triangle 45,x=0.3cm,y=0.3cm]
				\draw (0,0)--(0,1);
				\draw (0,1)--(-3,4);
				\draw (0,1)--(3,4);
				\draw (2,3)--(1,4);
				\draw (-2,3)--(-2,4);
				\draw (-2,3)--(-1,4);
				\draw (-2,3)--(-3,4);
				\draw[thick,orange,double] (-2,4)--(-2,3)--(0,1)--(0,0);
				\filldraw[color=white] (1,0) rectangle (-1,-0.1);
				\filldraw[color=white] (-2.2,4) rectangle (-1.8,4.2);
		\end{tikzpicture}} \rightarrow} \begin{tikzpicture}
		$ \left( \Y,\YY \right)$
	\end{tikzpicture}
\]
\end{Eg}
This notion arises from~\cite[Definition 2.2 and Example 2.3]{Eclair}, where a complete bialgebra called $\TSym$ is built with a coproduct based on the lightning decomposition of a tree. It also appears in~\cite{Tridend+co} where lightings are called \emph{cutting paths} in Section 7.1.1.
\begin{Not}
	Let $t$ be a tree and $m$ be an element of $\IEM{1}{\nf(t)}$, we denote by $({^mt},t^m)$ the lightning decomposition of $t$ on the leaf number $m$.
\end{Not}

Let $t$ be a tree of $\mathcal{A}$. Let us consider it as an element of $\mathcal{A}^{\circledast}$ through the pairing~\eqref{eq:pairing}. Let~$u$,~$v$ be trees of $\mathcal{A}$ and we denote
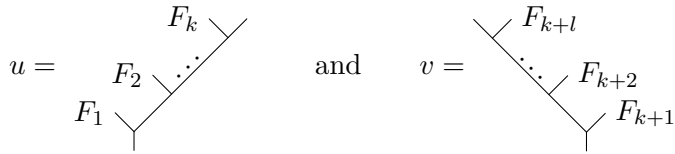
\begin{align*}
	& u=\peignedroit{1}{2}{k} \qquad \text{and} \qquad v=\peignegauche{k+1}{k+2}{k+l}
\end{align*}
Then
\begin{align*}
	\left\langle \Delta(t),u\otimes v \right\rangle &=\left\langle t, u*v \right\rangle
	= \sum_{\sigma\in \batc(k,l)} \left\langle t,\sigma(u,v) \right\rangle \\
	&=\begin{cases}
	1 &\text{if } \exists \sigma\in\batc(k,l) \text{ such that } t=\sigma(u,v), \\
	0 &\text{else}.
	\end{cases}
\end{align*}

\begin{Rq}
	By Remark $\ref{battage}$, if such a quasi-shuffle exists, it is necessarily unique.
\end{Rq}
As a consequence, if this quantity is non-zero, this implies that $t$ has the following form:
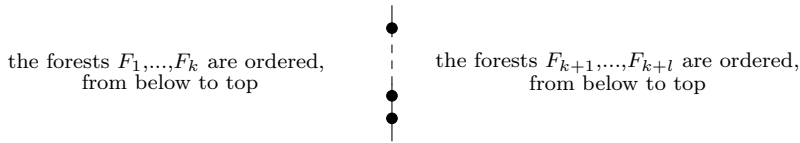
\[
	\begin{tikzpicture}[line cap=round,line join=round,>=triangle 45,x=0.3cm,y=0.3cm]
	\draw (0,0)--(0,2.5);
	\draw[dashed] (0,2.5)--(0,5);
	\draw (0,5)--(0,6);
	\filldraw [black] (0,1) circle (2pt);
	\filldraw [black] (0,2) circle (2pt);
	\filldraw [black] (0,5) circle (2pt);
	\draw (-10,3) node{$\substack{\text{the forests } F_1,\dots,F_k \text{ are ordered,} \\ \text{ from below to top}}$};
	\draw (10,3) node{$\substack{\text{the forests } F_{k+1},\dots,F_{k+l} \text{ are ordered,} \\ \text{from below to top}}$};
	\end{tikzpicture}
\]

If we use a lightning on the $m$-th leaf of the tree $t$ with $m=\nf(u)$, that is the ladder that we see on the drawing above, we find
\begin{equation}
	^mt=u \qquad \text{and} \qquad t^m=v.	\label{lienclair}
\end{equation}
 Moreover, all the elements of the form $\sigma(u,v)$ for a certain $\sigma\in\batc(k,l)$, satisfy
\[
 ^m\sigma(u,v)=u \qquad \text{and} \qquad \sigma(u,v)^m=v,
\]
where $m=\nf(u).$
And these trees are the only ones with this property. In fact, let us consider $t$ a tree and $m\in\N\setminus\{0\}$ such that $^mt=u$ and $t^m=v$. Then following the road of the lightning, we have the following description of $t$:
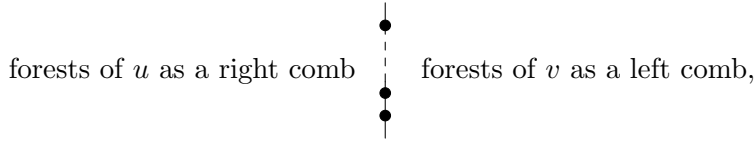
 \[
 	\begin{tikzpicture}[line cap=round,line join=round,>=triangle 45,x=0.3cm,y=0.3cm]
 \draw (0,0)--(0,2.5);
 \draw[dashed] (0,2.5)--(0,5);
 \draw (0,5)--(0,6);
 \filldraw [black] (0,1) circle (2pt);
 \filldraw [black] (0,2) circle (2pt);
 \filldraw [black] (0,5) circle (2pt);
 \draw (9,3) node{forests of $v$ as a left comb,};
 \draw (-9,3) node{forests of $u$ as a right comb};
 \end{tikzpicture}
 \]
where these forests appear ``in the good order''. Consequently, we deduce
\begin{Prop}\label{Prop:Coproduitdual}
	The coproduct of $\mathcal{A}^{\circledast}$ is given for all tree $t$ different from $|$ by
	\begin{equation*}
		\Delta(t)=\sum_{m=1}^{\nf(t)} {^mt}\otimes t^m.
	\end{equation*}
We get combinatorial interpretations of the pieces of coproduct:
\begin{enumerate}\itemsep=0pt
	\item[$1.$] We interpret $\Delta_{\prec}(t)$ as the sum of $^mt\otimes t^m$
	such that the lightning starting at the leaf number $m$ of $t$ touches the right-most exterior branch of $t$.
	\item[$2.$] The terms of $\Delta_{\succ}(t)$ is the sum of these terms where the lightning touches the left-most exterior branch of $t$.
	\item[$3.$] The terms of $\Delta_{\bullet}(t)$ are those where the lightning does not touch any of the exterior branches.
\end{enumerate}
\end{Prop}
\begin{Rq}\label{rq:coproduit}
		 It is the coproduct from $\TSym= \bigoplus_{n\geq 0}\K PT_n$ described in~\cite[Definition 4.1]{Eclair}. In~\cite[Section 7.3]{Tridend+co}, we find an equivalent description of the coproducts with cutting paths.
\end{Rq}

The following results will be useful in the Section~\ref{genprim}.

\begin{Cor}
Let $a$, $b$, $c$, $d$ be four homogeneous elements of $\mathcal{A}$. We put $\deg(a)=\deg(c)$ and $\deg(b)=\deg(d)$. We suppose that $a*b=c*d$ for the product of $\mathcal{A}$. Then $a$ and $b$ are collinear and $c$ and $d$ are collinear.
\end{Cor}
\begin{proof}
Firstly, we consider four trees $a$, $b$, $c$, $d$. We denote $k\coloneqq\nf(a)=\nf(b)$ and $l\coloneqq\nf(c)=\nf(d)$. We define for any tree $t$ with $m$ leaves and for all $n\in\IEM{1}{m}$
\[
\Delta_n(t)=^nt\otimes t^n.
\]
As a consequence, denoting $\lambda_{k,l}$ the number of elements of $\batc(k,l)$ (which is not zero),
\begin{gather*}
	\Delta_k(a*b)=\lambda_{k,l}a\otimes b, \qquad \Delta_{k}(c*d)=\lambda_{k,l}c\otimes d.
\end{gather*}
As $a*b=c*d$ and $\lambda_{k,l}\neq 0$, we deduce that $a$ and $b$ are collinear and $c$ and $d$ are collinear.
Secondly, we consider $a= \sum_{i\in\IEM{1}{n}} \lambda_i a_i$, $b= \sum_{j\in\IEM{1}{m}} \beta_j b_j$, $c= \sum_{q\in\IEM{1}{r}} \gamma_q c_q$ and $d= \sum_{l\in\IEM{1}{s}} \delta_l d_l$ where the $\lambda_i$, $\beta_j$, $\gamma_q$, $\delta_l$ are all non zero scalars and $a_i$, $b_j$, $c_q$, $d_l$ are trees. Expanding the writings of~$a*b$ and $c*d$ and applying $\Delta_k$, we also obtain
\begin{gather*}
	\Delta_k(a*b)=\lambda_{k,l}a\otimes b, \qquad \Delta_{k}(c*d)=\lambda_{k,l}c\otimes d.
\end{gather*}
We also have the same conclusion as in the first case.
\end{proof}

Following the same ideas, we also have
\begin{Cor}
		Let $a$, $b$, $c$, $d$ be four homogeneous elements of $\mathcal{A}$. We put $\deg(a)=\deg(c)$, $\deg(b)=\deg(d)$. We also suppose $a\prec b=c\prec d$ or $a\cdot b=c\cdot d$ or $a\succ b=c\succ d$.
	Then $a$ and $b$ respectively $c$ and $d$ are collinear.
\end{Cor}

\begin{thm}
	\label{thm:inthom}
	Let $X\in \mathcal{A}_k\otimes \mathcal{A}_l$ for some pair of non-negative integers $(k,l)$. Let $X= \sum_{i\in I} a_i\otimes b_i$ be a minimal writing of $X$ for the number of terms. If there exists $\ltimes\in\{\prec,\cdot,\succ,*\}$ such that $ \sum_{i\in I} a_i\ltimes b_i=0$, then $X=0.$
\end{thm}

\begin{proof}
Let $\ltimes\in\{\prec,\cdot,\succ,*\}$ such that $\sum_{i\in I} a_i\ltimes b_i=0$.
 We define $\Delta_k$ as in the previous proof. As a consequence,
\begin{align*}
	0=\Delta_k(X)=\Delta_k\bigg( \sum_{i\in I} a_i\ltimes b_i \bigg)&=\lambda^{\ltimes}_{(k,l)}\sum_{i\in I} a_i\otimes b_i,
\end{align*}
where $\lambda_{(k,l)}^{\prec}$, $\lambda_{(k,l)}^{\bullet}$, $\lambda_{(k,l)}^{\succ}$ are respectively the number of elements of $\batc(k,l)$ such that $\sigma^{-1}\!(\{1\})\!=\{1\}$, $\sigma^{-1}(\{1\})=\{1,k+1\}$ and $\sigma^{-1}(\{1\})=\{k+1\}$. None of these numbers are equal to zero. But this sum is $0$. As a consequence, at least one family $(a_i)_{i\in I}$ or $(b_i)_{i\in I}$ is linearly dependant. However $ X=\sum_{i\in I}a_i\otimes b_i$ is a minimal writing. So $X=0$.
\end{proof}

\subsection[Description of the product of A\^{}\{circledast\}]{Description of the product of $\boldsymbol{\mathcal{A}^{\circledast}}$}

Let us remind the definitions of $\shuffle$ and $m_{\TSym}$ from~\cite[Definitions 2.9 and 4.1]{Eclair}.
\begin{defi}
	Let $t$ be a tree with $k$ leaves and consider $k$ trees $t_1,\dots, t_k$. We denote by $t\shuffle (t_1,\dots,t_k)$ the tree obtained by grafting $t_i$ on the $i$-th leaf of $t$ for each $i\in\IEM{1}{k}$.
	We can also use multiple lightnings on a single tree $t$. If $l$ is a non-negative integer, we denote by $\Light_l(t)\subseteq \mathcal{A}^{l+1}$ the set of all the decompositions of $t$ obtained using $l$ successive lightnings. Consider another tree $s$ such that $\deg(s)=r$. We define	
	\[
	m_{\TSym}(s\otimes t)\coloneqq \sum_{(t_1,\dots,t_{r+1})\in\Light_r(t)} s\shuffle (t_1,\dots, t_{r+1}).
	\]
\end{defi}
\begin{Egs}
	Here is an example of multiple lightnings:
	\begin{align*}
		&\raisebox{-0.3\height}{\begin{tikzpicture}[line cap=round,line join=round,>=triangle 45,x=0.2cm,y=0.2cm]
				\draw (0,0)--(0,1);
				\draw (0,1)--(-3,4);
				\draw (0,1)--(3,4);
				\draw (2,3)--(1,4);
				\draw (-2,3)--(-2,4);
				\draw (-2,3)--(-1,4);
				\draw (-2,3)--(-3,4);
				\draw (-0.9,4) node[above]{\textcolor{orange}{$\lightning$}};
				\draw (-1.9,4) node[above]{\textcolor{red}{$\lightning$}};
				\draw[thick,red] (-1,4)--(-2,3)--(0,1)--(0,0);
				\draw[thick,red] (-2,4)--(-2,3)--(0,1)--(0,0);
		\end{tikzpicture}} \rightarrow \Bigg( \Y,
		\raisebox{-0.3\height}{\begin{tikzpicture}[line cap=round,line join=round,>=triangle 45,x=0.2cm,y=0.2cm]
				\draw (0,0) -- (0,1);
				\draw (0,1) -- (2,3);
				\draw (0,1) -- (-2,3);
				\draw (1.33,2.33) -- (0.66,3);
				\draw (-1.33,2.33) -- (-0.66,3);
				\draw (-0.66,3.1) node[above]{\textcolor{orange}{$\lightning$}};
				\draw[thick,red] (-0.66,3) --(-1.33,2.33) -- (0,1) -- (0,0);
		\end{tikzpicture}} \Bigg)\rightarrow \left( \Y,\Y,\balaisd \right), \\
		&\Light_2\left(\YY\right) =\mzleft{a}{\{}{\left(|,|,\YY\right), \left(|,\YY,|\right), \left(\YY,|,|\right),\big(|,\Y,\balaisd\big), \big(|,\balaisg,\Y\big),
		} \\
			&\hphantom{\Light_2\left(\YY\right) =}{} \quad \mzright{a}{\big(\Y,\balaisd,|\big),\big(\balaisg, \Y,|\big),\big( \Y,|,\balaisd\big),\big(\balaisg,|,\Y\big),\big(\Y,\Y,\Y\big)
			}{\}} .
	\end{align*}
	Here is an example for the product:
	\begin{align*}
	&	\Y\shuffle \big(\Y,\Y\big)=\YY, \\
	&m_{\TSym}\left(\balais,\YY\right)=\raisebox{-0.5\height}{\begin{tikzpicture}[x=0.1cm,y=0.1cm]
		\draw (0,0)--(0,7);
		\draw (0,1) --(6,7);
		\draw (0,1) -- (-6,7);
		\draw (5,6) -- (4,7);
		\draw (3.5,4.5) -- (1,7);
		\draw (2,6) -- (3,7);
	\end{tikzpicture}}+\raisebox{-0.5\height}{\begin{tikzpicture}[x=0.15cm,y=0.15cm]
	\draw (0,0)--(0,3);
	\draw (0,1) --(4,5);
	\draw (0,1) -- (-4,5);
	\draw (0,3) -- (2,5);
	\draw (0,3) -- (-2,5);
	\draw (1.2,4.2) -- (0.6,5);
	\draw (-1.2,4.2) -- (-0.6,5);
\end{tikzpicture}}+\raisebox{-0.5\height}{\begin{tikzpicture}[x=0.1cm,y=0.1cm]
\draw (0,0)--(0,7);
\draw (0,1) --(6,7);
\draw (0,1) -- (-6,7);
\draw (-5,6) -- (-4,7);
\draw (-3.5,4.5) -- (-1,7);
\draw (-2,6) -- (-3,7);
\end{tikzpicture}}+\raisebox{-0.5\height}{\begin{tikzpicture}[x=0.1cm,y=0.1cm]
\draw (0,0)--(0,6);
\draw (0,1) --(6,7);
\draw (0,1) -- (-6,7);
\draw (0,6) -- (-1,7);
\draw (0,6) -- (1,7);
\draw (4,5) -- (2,7);
\draw (5,6) -- (4,7);
\end{tikzpicture}}+\raisebox{-0.5\height}{\begin{tikzpicture}[x=0.1cm,y=0.1cm]
\draw (0,0)--(0,5);
\draw (0,1) --(6,7);
\draw (0,1) -- (-6,7);
\draw (5,6) -- (4,7);
\draw (0,5) -- (2,7);
\draw (0,5) -- (-2,7);
\draw (-1,6) -- (0,7);
\end{tikzpicture}} \\
&\hphantom{m_{\TSym}\left(\balais,\YY\right)=}{}+\raisebox{-0.5\height}{\begin{tikzpicture}[x=0.1cm,y=0.1cm]
		\draw (0,0)--(0,5);
		\draw (0,1) --(6,7);
		\draw (0,1) -- (-6,7);
		\draw (-5,6) -- (-4,7);
		\draw (0,5) -- (2,7);
		\draw (0,5) -- (-2,7);
		\draw (1,6) -- (0,7);
		\draw (1,6) -- (2,7);
\end{tikzpicture}}+\raisebox{-0.5\height}{\begin{tikzpicture}[x=0.1cm,y=0.1cm]
\draw (0,0)--(0,6);
\draw (0,1) --(6,7);
\draw (0,1) -- (-6,7);
\draw (0,6) -- (1,7);
\draw (0,6) -- (-1,7);
\draw (-4,5) -- (-2,7);
\draw (-5,6) -- (-4,7);
\end{tikzpicture}}+\raisebox{-0.5\height}{\begin{tikzpicture}[x=0.1cm,y=0.1cm]
\draw (0,0)--(0,6);
\draw (0,1) --(6,7);
\draw (0,1) -- (-6,7);
\draw (5,6) -- (4,7);
\draw (0,6) -- (1,7);
\draw (0,6) -- (-1,7);
\draw (4,5) -- (2,7);
\draw (5,6) -- (4,7);
\end{tikzpicture}}+\raisebox{-0.5\height}{\begin{tikzpicture}[x=0.1cm,y=0.1cm]
\draw (0,0)--(0,6);
\draw (0,1) --(6,7);
\draw (0,1) -- (-6,7);
\draw (-5,6) -- (-4,7);
\draw (0,6) -- (1,7);
\draw (0,6) -- (-1,7);
\draw (-4,5) -- (-2,7);
\end{tikzpicture}}+\raisebox{-0.5\height}{\begin{tikzpicture}[x=0.1cm,y=0.1cm]
\draw (0,0)--(0,5.33);
\draw (0,1) --(6,7);
\draw (0,1) -- (-6,7);
\draw (-4.33,5.33) -- (-2.66,7);
\draw (0,5.33) -- (1.66,7);
\draw (0,5.33) -- (-1.66,7);
\draw (4.33,5.33) -- (2.66,7);
\end{tikzpicture}}.
\end{align*}
\end{Egs}
\begin{Not}
	Consider a vector space $V$ and a linear map $m\colon V\otimes V \to V$, we denote $m^{\op}$ the linear map from $V\otimes V$ to $V$ such that for all $a\otimes b\in V\otimes V$
	\[
	m^{\op}(a\otimes b)=m(b\otimes a).
	\]
	If $(A,m,1)$ is an algebra, we denote $A^{\op}$ the algebra $(A,m^{\op},1)$.
\end{Not}
\begin{Prop}
	The product $m$ of $\mathcal{A}^\circledast$ is the product of $m_{\TSym}^{\op}$.
\end{Prop}
\begin{proof}
	Let $(u,v)\in T_n\times T_m\subset \mathcal{A}^{\circledast}\times \mathcal{A}^{\circledast}$ for $m$ and $n$ some non-negative integers. Let $w$ be a~tree of $\mathcal{A}$. In this case, thanks to Theorem $\ref{coproduit}$
	\begin{align*}
		& \langle m(u\otimes v), w \rangle= \langle u\otimes v, \Delta(w) \rangle
		 =\sum_{c \text{ admissible cut of }w} \delta_{u\otimes v}(G^c(w)\otimes R^c(w)).
	\end{align*}
Therefore, this quantity is not equal to zero \ssi{} $w=v\shuffle (w_1,\dots, w_{\deg(v)+1})$ and $u$ occurs in $G^c(w)$ with $c$ the cut of $w$ such that $R^c(w)=v$.	
	But $G^c(w)=w_1*\dots*w_{\deg(v)+1}$, by Theorem~\ref{produit} and Remark~\ref{lienclair}, we get that $G^c(w)$ is the sum of all the trees which gives $(w_1,\dots,w_{\deg(v)+1})$ after $\deg(v)$ well-placed lightnings. As a consequence, we have
	\begin{align*}
		m(u\otimes v)&=\sum_{u\mapsto (u_1,\dots,u_{\deg(v)+1})} v\shuffle(u_1,\dots, u_{\deg(v)+1}) \\
		&= m_{\TSym}(v\otimes u)=m_{\TSym}^{\op}(u\otimes v). \tag*{\qed}
	\end{align*}
\renewcommand{\qed}{}
\end{proof}

Thanks to the previous proposition, Proposition~\ref{Prop:Coproduitdual} and Remark~\ref{rq:coproduit}, we deduce that
\begin{thm}\label{thm:dual}
	The bialgebras $\mathcal{A}^{\circledast}$ and $\TSym^{\op}$ are the same.
\end{thm}
\begin{Rq}
	We can give $\TSym$ a $(1,3)$-dendriform bialgebra structure given in Proposition~\ref{Prop:Coproduitdual}.
\end{Rq}

\begin{Rq}
	With multiple generators, we have $A(\mathcal{D})^{\circledast}=\TSym(\mathcal{D})^{\op}.$ In the bialgebra $\TSym(\mathcal{D})$ its coproduct is an expansion of the lightning decomposition over elements of \linebreak $\bigoplus_{n\geq 0} T^g_n(T(\mathcal{D}))$. Let $t\in T^g_n(T(\mathcal{D}))$. We consider a lightning on one of its leaves. If a node of $t$ only appears on the left tree or only on the right one, then its decoration stays unchanged. Otherwise denoting $d_1\dots d_k$ its decoration and $l$ the edge number from which the lightning arrives, the decoration of this node on the left tree is $d_1\dots d_{l-1}$ and on the right tree it is $d_l\dots d_k$.
	\begin{Eg}
		We consider the following tree:\begin{equation*}
			t=\raisebox{-0.5\height}
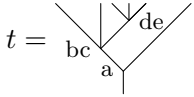
{\begin{tikzpicture}[line cap=round,line join=round,>=triangle 45,x=0.3cm,y=0.3cm]
					\draw (0,0)--(0,1);
					\draw (0,1)--(3,4);
					\draw (0,1)--(-3,4);
					\draw (-1,2)--(-1,4);
					\draw (-1,2)--(1,4);
					\draw (0.25,3.25)--(0.25,4);
					\draw (0.25,3.25)--(-0.5,4);
					\draw (0.25,3.25)--(1,4);
					\draw (0,1) node[below,left] {\footnotesize{a}};
					\draw (-1,2) node[left] {\footnotesize{bc}};
					\draw (0.25,3.25) node[right] {\footnotesize{de}};
			\end{tikzpicture}}
		\end{equation*}
	Considering the lightning coming from the fourth leaf, we get
	\begin{align*}
				\raisebox{-0.5\height}{\begin{tikzpicture}[line cap=round,line join=round,>=triangle 45,x=0.3cm,y=0.3cm]
				\draw[orange,thick] (0,0)--(0,1);
				\draw (0,1)--(3,4);
				\draw (-1,2)--(-3,4);
				\draw[orange,thick] (0,1) -- (-1,2);
				\draw (-1,2)--(-1,4);
				\draw[orange,thick] (-1,2)--(0.25,3.25);
				\draw[orange,thick] (0.25,3.25)--(0.25,4);
				\draw (0.25,3.25)--(-0.5,4);
				\draw (0.25,3.25)--(1,4);
				\draw (0,1) node[below,left] {\footnotesize{a}};
				\draw (-1,2) node[left] {\footnotesize{bc}};
				\draw (0.25,3.25) node[right] {\footnotesize{de}};
		\end{tikzpicture}}\rightarrow
		\bigg(\raisebox{-0.5\height}{\begin{tikzpicture}[line cap=round,line join=round,>=triangle 45,x=0.3cm,y=0.3cm]
				\draw (0,0)--(0,3);
				\draw (0,1)--(-2,3);
				\draw (1.33,2.33)--(0.66,3);
				\draw (0,1)--(2,3);
				\draw (0,1) node[left] {\footnotesize{bc}};
				\draw (1.33,2.33) node[right] {\footnotesize{d}};
		\end{tikzpicture}},\raisebox{-0.5\height}{\begin{tikzpicture}[line cap=round,line join=round,>=triangle 45,x=0.4cm,y=0.4cm]
				\draw (0,0)-- (0,1);
				\draw (0,1)-- (-1,2);
				\draw (-0.5,1.5)-- (0,2);
				\draw (0,1)-- (1,2);
				\draw (0,1) node[left] {\footnotesize{a}};
				\draw (-0.5,1.5) node[left] {\footnotesize{e}};
		\end{tikzpicture}}\bigg).
	\end{align*}
\end{Eg}
\end{Rq}

\section[(3,2)-dendriform bialgebras]{$\boldsymbol{(3,2)}$-dendriform bialgebras}
\subsection{Definition}

Let $(\mathcal{A},\prec,\succ,\cdot)$ be a tridendriform algebra.
 The definitions of \emph{dendriform algebras}, \emph{dendriform coalgebras}, \emph{dendriform bialgebras}, \emph{codendriform bialgebras} and \emph{bidendriform bialgebras} (that we will call $(2,2)$-dendriform bialgebras) can be found in the article~\cite{bidend}.

 \begin{defi}\label{deftroisdeux}
 	A \emph{$(3,2)$-dendriform bialgebra} is a sextuple $(A,\prec,\cdot,\succ,\Delta_{\leftarrow},\Delta_{\rightarrow})$ such that
 	\begin{itemize}\itemsep=0pt
 		\item $(A,\prec,\cdot,\succ)$ is an augmented tridendriform algebra.
 		\item $(A,\Delta_{\leftarrow},\Delta_{\rightarrow})$ is an augmented dendriform coalgebra.
 		\item The following relations are satisfied for all $a,b\in A$:
 		\begin{gather}
 			\Delta_{\leftarrow}(a\cdot b) =a'b'_{\leftarrow}\otimes a''\cdot b''_{\leftarrow}+b'_{\leftarrow}\otimes a\cdot b''_{\leftarrow}, \label{troisdeuxrel1}\\
 			\Delta_{\rightarrow}(a\cdot b) =a'b'_{\rightarrow}\otimes a''\cdot b''_{\rightarrow}+a'\otimes a''\cdot b+b'_{\rightarrow}\otimes a\cdot b''_{\rightarrow}, \label{troisdeuxrel2}\\
 			\Delta_{\leftarrow}(a\prec b) =a'b'_{\leftarrow}\otimes a''\prec b''_{\leftarrow}+a'b\otimes a'' +b'_{\leftarrow}\otimes a\prec b''_{\leftarrow}+b\otimes a, \label{troisdeuxrel3}\\
 			\Delta_{\rightarrow}(a\prec b) =a'b'_{\rightarrow}\otimes a''\prec b''_{\rightarrow}+a'\otimes a''\prec b +b'_{\rightarrow}\otimes a\prec b''_{\rightarrow},
 			\label{troisdeuxrel4}\\
 			\Delta_{\rightarrow}(a\succ b) =a'b'_{\rightarrow}\otimes a''\succ b''_{\rightarrow}+a'\otimes a''\succ b +b'_{\rightarrow}\otimes a\succ b''_{\rightarrow} \nonumber\\ \hphantom{\Delta_{\rightarrow}(a\succ b) =}{}+ab'_{\rightarrow}\otimes b''_{\rightarrow}+a\otimes b,
 	\label{troisdeuxrel5}		\\
 			\Delta_{\leftarrow}(a\succ b) =a'b'_{\leftarrow}\otimes a''\succ b''_{\leftarrow}+a'\otimes a''\succ b +ab'_{\leftarrow}\otimes b''_{\leftarrow}.
 			\label{troisdeuxrel6}
 		\end{gather}
 	\end{itemize}
 \end{defi}
\begin{Rq}
	The definition of $(3,2)$-dendriform bialgebra is made such that
$
	 (A,\prec,\cdot,\succ,\Delta_{\leftarrow},\allowbreak \Delta_{\rightarrow})$ is a $(3,2)$-dendriform bialgebra
		 $\iff$
		 $(A,\prec,\cdot+\succ,\Delta_{\leftarrow},\Delta_{\rightarrow})$ and $(A,\prec+\cdot,\succ,\Delta_{\leftarrow},\Delta_{\rightarrow})$
	 are $(2,2)$-dendriform bialgebras.
\end{Rq}
To check that this definition is consistent, we will give an example of such an object.

\subsection[Example of a (3,2)-dendriform bialgebra]{Example of a $\boldsymbol{(3,2)}$-dendriform bialgebra}
\subsubsection{Description}
To find an example, we consider the $(3,1)$-dendriform bialgebra given by $\mathcal{A}$. Our objective is to break the coproduct $\Delta$ into two pieces, denoted by $\Delta_{\rightarrow}$ and $\Delta_{\leftarrow}$, which will satisfy the desired relations.
Let us consider $t$ a tree and let us decorate the right-most leaf by a symbol $d$.
\begin{Eg}
	For \smash{$t=$\Y}, this gives \smash{\Yd}.
\end{Eg}

\begin{defi}
	Let $t$ be a tree of $\mathcal{A}$ different from $|$.
	We define $\Delta_{\leftarrow}(t)$ as the sum of the terms from $\Delta(t)$ which gets the decorated leaf by $d$ (that means the right-most leaf of $t$) on the left tensor, i.e., on $G^c(t)$. Then we delete the decoration.
	
	We define $\Delta_{\rightarrow}(t)$ as the sum of the terms from $\Delta(t)$ which gets the decorated leaf by $d$ on the right tensor, i.e., on $R^c(t)$. Then we delete the decoration.
	So, $\Delta(t)=\Delta_{\leftarrow}(t) +\Delta_{\rightarrow}(t).$
\end{defi}

\begin{Not}
	Let $t$ be a tree. We denote $t^d$ the tree $t$ which last leaf is decorated with $d$. As we do not talk about lightnings in this section, we use again this notation during this section.
\end{Not}

\begin{Prop}\label{Prop:troideuxEg}
	We consider $(\mathcal{A},\prec,\cdot,\succ)$ with its tridendriform algebra structure, we build the following coproducts:
	\begin{gather*}
		\tilde{\Delta}_{\leftarrow}(t)=\sum_{\substack{c \text{ \normalfont admissible cut of }t \\ \text{ \normalfont the right-most leaf of $t$ is cut}}} G^c(t)\otimes R^c(t)-t\otimes 1, \\
		\tilde{\Delta}_{\rightarrow}(t)=\sum_{\substack{c \text{ \normalfont admissible cut of }t \\ \text{ \normalfont the right-most leaf of $t$ is not cut}}} G^c(t)\otimes R^c(t) -1\otimes t.
	\end{gather*}
	Then $(\mathcal{A},\prec+\cdot,\succ,\Delta_{\leftarrow},\Delta_{\rightarrow})$ and $(\mathcal{A},\prec,\cdot+\succ,\Delta_{\leftarrow},\Delta_{\rightarrow})$ are bidendriform bialgebras. As a consequence, $(\mathcal{A},\prec,\cdot,\succ,\Delta_{\leftarrow},\Delta_{\rightarrow})$ is a $(3,2)$-dendriform bialgebra.
\end{Prop}
\begin{proof}
	We will now show that $(\mathcal{A},\prec,\cdot,\succ,\Delta_{\leftarrow},\Delta_{\rightarrow})$ satisfies the Definition~\ref{deftroisdeux}.
	The first constraint is satisfied.
	For the second one, we need to check the three following equalities:{\samepage
	\begin{gather}
		\big(\tilde{\Delta}_{\leftarrow}\otimes {\rm Id}\big)\circ \tilde{\Delta}_{\leftarrow}=\big({\rm Id}\otimes \tilde{\Delta}\big)\circ \tilde{\Delta}_{\leftarrow}, \label{codendri1} \\
		\big(\tilde{\Delta}_{\rightarrow}\otimes {\rm Id}\big)\circ \tilde{\Delta}_{\leftarrow}=\big({\rm Id} \otimes \tilde{\Delta}_{\leftarrow}\big)\circ \tilde{\Delta}_{\rightarrow}, \label{codendri2} \\
		\big(\tilde{\Delta}\otimes {\rm Id}\big)\circ \tilde{\Delta}_{\rightarrow}=\big({\rm Id}\otimes \tilde{\Delta}_{\rightarrow}\big)\circ \tilde{\Delta}_{\rightarrow}. \label{codendri3}
	\end{gather}

}\pagebreak

\noindent
	We prove each case by following the destination of the decorated leaf with $d$ and using that $\tilde{\Delta}=\tilde{\Delta}_{\leftarrow}+\tilde{\Delta}_{\rightarrow}$ to obtain those equalities.
	More precisely, let us consider a tree $t$ of $\mathcal{A}$. In order to show $\eqref{codendri1}$, we look for all the terms of $\big(\tilde{\Delta}\otimes {\rm Id}\big)\circ \tilde{\Delta}\big(t^d\big)$ and $\big({\rm Id}\otimes \tilde{\Delta}\big)\circ \tilde{\Delta}\big(t^d\big)$
	where the decorated leaf is on the left tensor. In the first term, we find $\big(\tilde{\Delta}_{\leftarrow}\otimes {\rm Id}\big)\circ \tilde{\Delta}_{\leftarrow}\big(t^d\big)$, then in the second, we find $\big({\rm Id}\otimes \tilde{\Delta}\big)\circ \tilde{\Delta}_{\leftarrow}\big(t^d\big)$.
	Finally, the coassociativity property gives that $\big(\tilde{\Delta}_{\leftarrow}\otimes {\rm Id}\big)\circ \tilde{\Delta}_{\leftarrow}\big(t^d\big)=\big({\rm Id}\otimes \tilde{\Delta}\big)\circ \tilde{\Delta}_{\leftarrow}\big(t^d\big)$.
	We proceed the same way with~\eqref{codendri2} and~\eqref{codendri3}.
	
	For the last condition, we remind that for all $a,b\in \mathcal{A}$ and for all $\ltimes\in\{ \prec,\cdot,\succ\}$
	\begin{align*}
		\Delta(a\ltimes b)
		&= \Delta(a)\ltimes \Delta(b) = (\Delta_{\rightarrow}(a)+\Delta_{\leftarrow}(a))\ltimes (\Delta_{\rightarrow}(b)+\Delta_{\leftarrow}(b)) \\
		&= \Delta_{\leftarrow}(a)\ltimes \Delta_{\leftarrow}(b)+\Delta_{\leftarrow}(a)\ltimes \Delta_{\rightarrow}(b)+
		\Delta_{\rightarrow}(a)\ltimes \Delta_{\leftarrow}(b)+
		\Delta_{\rightarrow}(a)\ltimes \Delta_{\rightarrow}(b).
	\end{align*}
As a consequence, identifying the terms which get the last leaf of $a\ltimes b$ at the left or at the right of the tensor, we obtain for all $\ltimes\in\{\prec,\cdot,\succ\}$
	\begin{align*}
		&\Delta_{\rightarrow}(a\ltimes b) =\Delta_{\leftarrow}(a)\ltimes \Delta_{\rightarrow}(b)+\Delta_{\rightarrow}(a)\ltimes \Delta_{\rightarrow}(b) =\Delta(a)\ltimes \Delta_{\rightarrow}(b), \\
		&\Delta_{\leftarrow}(a\ltimes b)
		 =\Delta_{\leftarrow}(a)\ltimes\Delta_{\leftarrow}(b)+\Delta_{\rightarrow}(a)\ltimes \Delta_{\leftarrow}(b) =\Delta(a)\ltimes \Delta_{\leftarrow}(b).
	\end{align*}
We need to check the six relations in the definition of $(3,2)$-dendriform bialgebra. I will only detail the computation for~\eqref{troisdeuxrel4}, the other verifications are left to the reader. Let $a,b\in \mathcal{A}$, then
	\begin{align*}
		\tilde{\Delta}_{\rightarrow}(a\prec b)
		={}&(1\otimes a + a\otimes 1+a'\otimes a'')\prec (1\otimes b+b'_{\rightarrow}\otimes b''_{\rightarrow})- 1\otimes a\prec b\\
		={}&1\otimes a\prec b+b'_{\rightarrow}\otimes a\prec b''_{\rightarrow}+ a'\otimes a''\prec b+a'*b'_{\rightarrow}\otimes a\prec b''_{\rightarrow}-1\otimes a\prec b \\
		={}& a'*b'_{\rightarrow}\otimes a''\prec b''_{\rightarrow}+a'\otimes a''\prec b+b'_{\rightarrow}\otimes a\prec b''_{\rightarrow}. \tag*{\qed}
	\end{align*}
\renewcommand{\qed}{}
\end{proof}

\subsubsection{Counting codendriform primitives}

In order to study the primitives of the $(3,2)$-dendriform bialgebra $\mathcal{A}$, we will use the results from the article~\cite{bidend}. Let us remind results from~\cite{bidend} without proofs.

\begin{defi}
	Let $(C,\Delta_{\leftarrow},\Delta_{\rightarrow})$ be a dendriform coalgebra. We define
	\begin{gather*}
		\Prim_{\leftarrow}(C) :=\big\{ c\in C \mid \tilde{\Delta}_{\leftarrow}(c)=0 \big\}, \\
		\Prim_{\rightarrow}(C) :=\big\{ c\in C \mid \tilde{\Delta}_{\rightarrow}(c)=0 \big\}, \\
		\Prim_{\rm Coass}(C) :=\big\{ c\in C \mid \tilde{\Delta}(c)=0 \big\},\\
		\Prim_{\rm Codend}(C) :=\big\{ c\in C \mid \tilde{\Delta}_{\leftarrow}(c)=\tilde{\Delta}_{\rightarrow}(c)=0 \big\}.
	\end{gather*}
\end{defi}

\begin{Rq}
	In particular, $\Prim_{\rm Codend}(C)=\Prim_{\leftarrow}(C)\cap \Prim_{\rightarrow}(C)\subseteq\Prim_{\rm Coass}(C).$
\end{Rq}
Let $(C,\Delta_{\leftarrow},\Delta_{\rightarrow})$ be a dendriform coalgebra. We introduce the following sets defined by induction:
	\begin{gather*}
		\mathcal{P}_C(0)=\{{\rm Id}_C\}, \\
		\mathcal{P}_C(1)=\{\Delta_{\leftarrow},\Delta_{\rightarrow}\}\subseteq \mathcal{L}(C,C^{\otimes 2}), \\
		\mathcal{P}_C(n)=\big\lbrace \big( {\rm Id}^{\otimes(i-1)}\otimes \Delta_{\leftarrow} \otimes {\rm Id}^{\otimes (n-i)} \big)\circ P \mid P\in\mathcal{P}_C(n-1), i\in\IEM{1}{n} \big\rbrace \\
		 \hphantom{\mathcal{P}_C(n)=}{}\cup \big\lbrace \big( {\rm Id}^{\otimes(i-1)}\otimes \Delta_{\rightarrow} \otimes {\rm Id}^{\otimes (n-i)} \big)\circ P \mid P\in\mathcal{P}_C(n-1), i\in\IEM{1}{n} \big\rbrace \! \subseteq\! \mathcal{L}\big(C,C^{\otimes (n+1)}\big).
	\end{gather*}
\begin{defi}
	We say that a dendriform coalgebra $C$ is \emph{connected} if for all $a\in C,$ there exists~$n_a\in\N$ such that for all $P\in\mathcal{P}_C(n_a), P(a)=0.$
\end{defi}

{\samepage With these definitions, we get
\begin{Lemme}
	The coalgebra $(\mathcal{A},\Delta_{\leftarrow},\Delta_{\rightarrow})$ is connected.
\end{Lemme}

}

\begin{proof}
	We notice that for all $\triangleright\in\{ \leftarrow,\rightarrow \}$, we have for all $n\in\N$
	\[
	\Delta_{\triangleright}(\mathcal{A}_n)\subseteq\sum_{k=1}^{n-1} \mathcal{A}_{n-k}\otimes \mathcal{A}_k.
	\]
	Therefore, choosing $t\in \mathcal{A}_n$, $p\in\N$ and $\delta\in\mathcal{P}_C(p)$, we get
	\[
	\delta(\mathcal{A}_n)\subseteq \sum_{\substack{k_1+\dots+k_{p+1}=n \\ k_1,\dots,k_{p}+1>0}} \mathcal{A}_{k_1}\otimes\dots\otimes \mathcal{A}_{k_{p+1}}.
	\]
	So, for all $p\geq n$ and for all $\delta\in\mathcal{P}_C(p)$, $\delta(\mathcal{A}_n)=(0).$
\end{proof}

Following~\cite{bidend}, we get

\begin{thm}
	Let $H$ be a $(2,2)$-dendriform bialgebra, which we assume to be connected as a~dendriform coalgebra. Then $H$ is generated by $\Prim_{\rm Codend}(H)$ as a dendriform algebra.
\end{thm}
\begin{proof}
	See~\cite[Theorem~21]{bidend}.
\end{proof}

This theorem also applies to $(\mathcal{A},\prec+\cdot, \succ,\Delta_{\leftarrow},\Delta_{\rightarrow})$.
So, $\mathcal{A}$ is generated by $\Prim_{\rm Codend}(\mathcal{A})$ with the operations $\prec+\cdot$, $\succ$.
In particular, $\mathcal{A}$ is generated by $\Prim_{\rm Codend}(\mathcal{A})$ as a tridendriform algebra. More generally, this gives

\begin{Prop}
	Let $H$ be a connected $($as a dendriform coalgebra$)$ $(3,2)$-dendriform bialgebra. Then $H$ is generated by $\Prim_{\rm Codend}(H)$ as a tridendriform algebra.
\end{Prop}
\begin{Rq}
	By~\cite{bidend}, we have
	\begin{align*}
		\mathcal{A}&=\Prim_{\rm Codend}(\mathcal{A})\oplus (\mathcal{A}\preceq \mathcal{A}+\mathcal{A}\succ\mathcal{A} ) \\
		&=\Prim_{\rm Codend}(\mathcal{A})\oplus (\mathcal{A}\prec \mathcal{A}+\mathcal{A}\succeq\mathcal{A} ) \\
		&=\Prim_{\rm Codend}(\mathcal{A})+ (\mathcal{A}\prec \mathcal{A}+\mathcal{A}\cdot \mathcal{A}+\mathcal{A}\succ\mathcal{A} ).
	\end{align*}
\end{Rq}
\begin{thm}	\label{thmcodend}
	Let $H$ be a graded $(2,2)$-dendriform bialgebra by $\N$ such that for all $n\in\N$, $H_n$~is finite dimensional and $H_0=(0)$. We consider the formal series
	\begin{gather*}
		P(X)=\sum_{n=1}^{+\infty}\dim(\Prim_{\rm Codend}(H)_n)X^n, \qquad R(X)=\sum_{n=1}^{+\infty}\dim(H_n)X^n.
	\end{gather*}
	Then $P(X)=\dfrac{R(X)}{(1+R(X))^2}$.
\end{thm}
\begin{proof}
	See~\cite[Corollary 37]{bidend}.
\end{proof}

\begin{Rq}
	We may ask ourselves why we want to compute codendriform primitives. We already have the generator \Y. But it is a generator for the tridendriform structure meanwhile $\Prim_{\rm Codend}(\mathcal{A})$ generates $\mathcal{A}$ as a dendriform algebra for both $(\prec+\cdot,\succ)$ and $(\prec,\cdot+\succ)$.
\end{Rq}
All the theorems cited above are valid for $(\mathcal{A}^+,\prec,\cdot,\succ,\Delta_{\leftarrow},\Delta_{\rightarrow})$.
Let us compute the formal series:
\[
R(X)=\sum_{n=1}^{+\infty}\dim(\mathcal{A}_n)X^n.
\]
\begin{defi}[non-redundant bracketing]
	Let $n\in \N$.
	Let $w$ be a word over the alphabet~$\{|\}$ of $n+1$ symbols.
	We call \emph{non-redundant bracketing} of $w$ all insertions of pairs of parenthesis in~$w$ which do not contain subwords of the form $(|)$ and such that $w$ is not between brackets.
	We will denote the set of non-redundant bracketing of any word $w$ by $\mathcal{P}_n$. The set of non-redundant bracketings of all words is denoted by $\mathcal{P}$.
\end{defi}
\begin{Rq}
	We have $\mathcal{P}=\bigcup_{n\geq 0} \mathcal{P}_n$.
	Moreover, we will denote $\mathcal{T}:=\bigcup_{n\geq 0} T_n$.
\end{Rq}
\begin{Eg}
	Let us take the word $w=|||$, the non-redundant bracketings of $w$ are~$|(||)$,~$(||)|$ and~$|||$. So $|\mathcal{P}_2|=3.$
\end{Eg}

It is known that $\dim(\mathcal{A}_n)=|\mathcal{P}_n|$. Referring to~\cite{OEIS}, we know that $|\mathcal{P}_n|=a_n$ where $a_n$ is the $n$-th second Schr\"oder's problem number.
The numbers $a_n$ are called \emph{small Schr\"oder's numbers} or super Catalan numbers.
With this reference, we also have
\begin{gather*}
	\forall n\in\N, \, n\geq 2,\, (n+1)a_n=(6n-3)a_{n-1}-(n-2)a_{n-2} \qquad \text{and} \qquad a_0=1, \, a_1=1.
\end{gather*}

Thanks to this, with the help of a computer we get
\begin{gather*}
	R(X)=X+3X^2+11X^3+45X^4+197X^5+903X^6+4279X^7+20793 {{X}^{8}} \\
\hphantom{R(X)=}{} +103049 {{X}^{9}}+518859 {{X}^{10}} +2646723 {{X}^{11}}+13648869 {{X}^{12}}+71039373 {{X}^{13}} \\
\hphantom{R(X)=}{} +372693519 {{X}^{14}}+1968801519 {{X}^{15}}+10463578353 {{X}^{16}}+55909013009 {{X}^{17}} \\
\hphantom{R(X)=}{} +300159426963 {{X}^{18}}+1618362158587 {{X}^{19}}+8759309660445 {{X}^{20}}+\cdots, \\
	P(X)=X+{{X}^{2}}+2 {{X}^{3}}+6 {{X}^{4}}+22 {{X}^{5}}+90 {{X}^{6}}+394 {{X}^{7}}+1806 {{X}^{8}}+8558 {{X}^{9}}\\
\hphantom{P(X)=}{} +41586 {{X}^{10}}+206098 {{X}^{11}}+1037718 {{X}^{12}}+5293446 {{X}^{13}}+27297738 {{X}^{14}} \\
\hphantom{P(X)=}{} +142078746 {{X}^{15}} +745387038 {{X}^{16}}+3937603038 {{X}^{17}}+20927156706 {{X}^{18}}\\
\hphantom{P(X)=}{} +111818026018 {{X}^{19}}+600318853926 {{X}^{20}}+\cdots.
\end{gather*}

Other properties of small Schr\"oder's numbers are given in~\cite{OEIS}.

Looking at the first terms of $P$, we suspect that $P$ is the series of \emph{large Schr\"oder's numbers} denoted by $(A_n)_{n\in\N}$ defined by $A_0=0$, $A_1=1=A_2$, then for all $n\geq 3$, $A_{n}=2a_{n-2}$.

\begin{Rq}\label{Rq:series}
	In particular, we want to show
	\[
	P(X)=X+X^2+2X^2R(X).
	\]
\end{Rq}

Looking once more at~\cite{OEIS}, we find that the formal series of $R(X)$ is
\begin{equation*}
	1+R(X)=\frac{1+X-\sqrt{1-6X+X^2}}{4X}.
\end{equation*}
Referring to~\cite{OEIS} and using Theorem~\ref{thmcodend}, we obtain
\begin{Prop}\label{nbcodend}
	\[
	P(X)=\sum_{n=0}^{+\infty} A_nX^n
	\]
	where $A_0=0$, $A_1=A_2=1$ and for all $n\geq 3$, $A_n=2a_{n-2}.$
\end{Prop}
\begin{proof}
	To prove this proposition, we provide an intermediate computation
	\begin{align}
		\frac{R(X)}{1+R(X)}
		&=\frac{1-3X-\sqrt{1-6X+X^2}}{1+X-\sqrt{1-6X+X^2}} \nonumber \\
		&=\frac{\big(1-3X-\sqrt{1-6X+X^2}\big)\big(1+X+\sqrt{1-6X+X^2}\big)}{\big(1+X-\sqrt{1-6X+X^2}\big)\big(1+X+\sqrt{1-6X+X^2}\big)} \nonumber \\
		&=\frac{4X-4X^2-4X\sqrt{1-6X+X^2}}{8X} \nonumber \\
		&=X+\frac{1}{2}\big(1-3X-\sqrt{1-6X+X^2}\big) \nonumber \\
		&=X+2XR(X). \label{cal:PrimCoass}
	\end{align}
Then
 \begin{align*}
	&\frac{R(X)}{(1+R(X))^2}=\frac{X}{1+R(X)}+2X\frac{R(X)}{1+R(X)} =\frac{X}{1+R(X)}+2X^2+4X^2R(X).
\end{align*}
Now we compute the first term of the right part of this equality:
\begin{align*}
	\frac{X}{1+R(X)}&=\frac{4X^2}{1+X-\sqrt{1-6X+X^2}} \\
					&=\frac{4X^2\big(1+X+\sqrt{1-6X+X^2}\big)}{\big(1+X-\sqrt{1-6X+X^2}\big)\big(1+X+\sqrt{1-6X+X^2}\big)} \\
					&=\frac{4X^2+4X^3+4X^2\sqrt{1-6X+X^2}}{8X} \\
					&=\frac{X}{2}\big(1+X+\sqrt{1-6X+X^2}\big) \\
					&=\frac{X}{2}(1+R(X))+X\sqrt{1-6X+X^2}.
\end{align*}
Finally,
\begin{align*}
	\frac{R(X)}{(1+R(X))^2}
	={}&2X^2(1+R(X))+X\sqrt{1\!-\!6X\!+\!X^2}+2X^2+X\big(1-3X-\sqrt{1\!-\!6X\!+\!X^2}\big) \\
	={}& X+X^2+2X^2R(X). \tag*{\qed}
\end{align*}
\renewcommand{\qed}{}
\end{proof}

As a consequence, we deduce the two following equalities:
\begin{align*}
	&P(X)=\frac{3}{2}X+\frac{3}{2}X^2-\frac{X}{2}\sqrt{1-6X+X^2}, \qquad
	R(X)=\frac{P(X)}{2X^2}-\frac{1}{2X}-\frac{3}{2}.
\end{align*}

\subsubsection{Codendriform primitives generation}\label{genprim}

Now we exactly know the dimensions of codendriform primitives. For degree 1 and 2, we have
\begin{gather*}
	\Prim_{\rm Codend}(\mathcal{A})_1=\big\langle \Y \big\rangle, \qquad
	\Prim_{\rm Codend}(\mathcal{A})_2=\big\langle \balais \big\rangle.
\end{gather*}
\begin{Rq}
	Note that $\balaisg -\balaisd$ is an element of $\Prim_{\rm Coass}(\mathcal{A})$.
	However,
	\begin{align*}
		\Delta_{\leftarrow}\left(\balaisg-\balaisd\right)=0-\Y \otimes \Y.
	\end{align*}
	So, it is not an element of $\Prim_{\rm Codend}(\mathcal{A})$.
\end{Rq}
From equations \eqref{troisdeuxrel1}--\eqref{troisdeuxrel6} we deduce the following equivalences:

\begin{Prop}
	For all $a,b\in \mathcal{A}$, we have
	\begin{gather*}
	(b\in\Prim_{\rightarrow}(\mathcal{A}) \  \text{\normalfont and} \  a\in\Prim_{\rm Coass}(\mathcal{A})) \iff \Delta_{\rightarrow}(a\cdot b)=0, \\
	\hphantom{(b\in\Prim_{\rightarrow}(\mathcal{A})\  \text{\normalfont and} \ }{} b\in\Prim_{\leftarrow}(\mathcal{A}) \iff \Delta_{\leftarrow}(a\cdot b)=0.
	\end{gather*}
\end{Prop}
\begin{proof}
	The $\implies$ implications come straightforward from the codendriform relations.
	For the other implication of the second equivalence, let us consider $a,b\in \mathcal{A}$ such that $\Delta_{\leftarrow}(a\cdot b)=0$.
	Then we suppose by contradiction that $\Delta_{\leftarrow}(b)\neq 0$. Without loss of generality, we can suppose that the writings of $\Delta_{\leftarrow}(b)$ as a sum of tensors are minimal for the number of terms.
	
	\textit{Case 1}: $\tilde{\Delta}(a)=0$.
	Using codendriform relations, we know that
	\[
	\Delta_{\leftarrow}(a\cdot b)=b'_{\leftarrow}\otimes a\cdot b''_{\leftarrow}=0.
	\]
	Defining $\mathfrak{B}=\{ c\otimes d \mid c\otimes d \text{ appears in } \Delta_{\leftarrow}(a\cdot b) \text{ such that } \deg(c) \text{ is maximal}\}$, we get
	\[
	\Delta_{\leftarrow}(b)=\sum_{\mathfrak{B}} b'_{\leftarrow}\otimes b''_{\leftarrow}+\sum_{\text{the rest}} b'_{\leftarrow}\otimes b''_{\leftarrow}.
	\]
	Let also denote by $\mathfrak{B}_1$ and $\mathfrak{B}_2$ the respective projections of $\mathfrak{B}$ over its first and second component. Remark that in this case, if $b'_{\leftarrow}\otimes b''_{\leftarrow}$ appears in $\Delta_{\leftarrow}(b)$ such that $b'_{\leftarrow}\otimes b''_{\leftarrow}\in\mathfrak{B}$ then $b''_{\leftarrow}$ is minimal for the degree. For reasons of degrees, we must have
	\[
	\sum_{\mathfrak{B}} b'_{\leftarrow}\otimes a\cdot b''_{\leftarrow}=0.
	\]
	As the writings of $\Delta_{\leftarrow}(b)$ are supposed minimal, this implies that the family $(b'_{\leftarrow})$ is free and the same for $(b''_{\leftarrow})$.
Applying Theorem~\ref{thm:inthom}, we find that the family $(a\cdot b''_{\leftarrow})$ is free.
We complete those free families $(b'_{\leftarrow})_{b'_{\leftarrow}\in\mathfrak{B}_1}$ and $(a\cdot b''_{\leftarrow})_{ b''_{\leftarrow}\in\mathfrak{B}_2}$ into bases respectively of $T_{\deg(\mathfrak{B}_1)}$ and $T_{\deg(a)+\deg(\mathfrak{B}_2)}$. We consider $f$, $g$ elements of the dual basis verifying $f(b'_{\leftarrow})=1$ for one and only one element of~$\mathfrak{B}_1$ and $g(a\cdot b''_{\leftarrow})=1$ for one and only one element of $(a\cdot b''_{\leftarrow})_{b''_{\leftarrow}\in\mathfrak{B}_2}$. Choosing correctly $(f,g)$, we find
	\begin{align*}
		0=(f\otimes g)\bigg( \sum_{\mathfrak{B}} b'_{\leftarrow}\otimes a\cdot b'_{\leftarrow} \bigg)=1.
	\end{align*}
	This is a contradiction. As a consequence, we find $\Delta_{\leftarrow}(b)=0$.
	
	\textit{Case 2}: $\tilde{\Delta}(a)\neq 0$.
	Let us denote the set
	\[
	\mathfrak{A} = \big\{ c\otimes d\mid c\otimes d \text{ appears in }\tilde{\Delta}(a) \text{ and } c \text{ has maximal degree}\big\}.
	\] We define $\mathfrak{B}$ the same way as the previous case.
	We take back the notations for projections over first and second component of $\mathfrak{B}$ and we expand them for $\mathfrak{A}$. The tensors of $\Delta_{\leftarrow}(a\cdot b)$ from which the first term has maximal degree are
	\[
	\sum_{\mathfrak{A}}\sum_{\mathfrak{B}} a'*b'_{\leftarrow}\otimes a''\cdot b''_{\leftarrow}.
	\]
	Without loss of generality, we can suppose that the writings of $\Delta_{\leftarrow}(b)$ and $\tilde{\Delta}(a)$ are minimal for the number of terms.
		We check $(a'* b'_{\rightarrow})_{(a',b'_{\rightarrow})\in\mathfrak{A}_1\times \mathfrak{B}_1}$ and $(a''\prec b''_{\rightarrow})_{(a'',b''_{\rightarrow})\in\mathfrak{A}_2\times \mathfrak{B}_2}$ are free families of $\mathcal{A}$.
	For this, we will use the Theorem~\ref{thm:inthom}. In fact, we consider for all $(a',b')\in\mathfrak{A}_1\times \mathfrak{B}_1$ a scalar $\lambda_{(a',b'_{\rightarrow})}$ such that
	\begin{align*}
		\sum_{\mathfrak{A}_1\times \mathfrak{B}_1}\lambda_{(a',b'_{\rightarrow})}a'*b'_{\rightarrow}=0.
	\end{align*}
	By construction of $\mathfrak{A}_1\times \mathfrak{B}_1,$ hypotheses of Theorem~\ref{thm:inthom} are satisfied. As a consequence, $(a'*b'_{\rightarrow})_{(a',b'_{\rightarrow})\in\mathfrak{A}_1\times \mathfrak{B}_1}$ is a free family. Following the same method, we have that the family $(a''\prec b''_{\rightarrow})_{(a'',b''_{\rightarrow})\in\mathfrak{A}_2\times \mathfrak{B}_2}$ is also free.
	 We can complete each of these families into respective basis of $T_{\deg(\mathfrak{B}_1)+\deg(\mathfrak{A}_1)}$ and $T_{\deg(\mathfrak{B}_2)+\deg(\mathfrak{A}_2)}$.
	Using elements $f$, $g$ of the dual basis defined by ${f(a'*b'_{\leftarrow})=1}$ for one and only one element of $(a'*b'_{\leftarrow})_{a'\in\mathfrak{A}_1,b'_{\leftarrow}\in\mathfrak{B}_1}$ and $g(a''\cdot b''_{\leftarrow})_{a''\in\mathfrak{A}_2,b''_{\leftarrow}\in\mathfrak{B}_2}=1$ for one and only one element of $(a''\cdot b''_{\leftarrow})$.
	As a consequence, with a good choice for $(f,g)$, we get
	\begin{align*}
		0=(f\otimes g)\bigg( 	\sum_{\mathfrak{A}}\sum_{\mathfrak{B}} a'*b'_{\leftarrow}\otimes a''\cdot b''_{\leftarrow} \bigg)=1.
	\end{align*}
This is a contradiction. We then deduce that $\tilde{\Delta}(a)=0$ or $\Delta_{\leftarrow}(b)=0$.
	
	For the first equivalence, we use similar ideas.
\end{proof}

This allows us to say
\[
a\cdot b\in \Prim_{\rm Codend}(\mathcal{A}) \iff (b\in\Prim_{\rm Codend}(\mathcal{A}) \  \text{and} \  a\in\Prim_{\rm Coass}(\mathcal{A})).
\]
This defines a map
\begin{equation*}
	\theta_{\bullet}\colon \ \begin{cases}
		\Prim_{\rm Coass}(\mathcal{A})\otimes \Prim_{\rm Codend}(\mathcal{A}) \rightarrow \Prim_{\rm Codend}(\mathcal{A}), \\
		a\otimes b \mapsto a\cdot b.
	\end{cases}
\end{equation*}

Using similar arguments, we also prove
\begin{Prop}
	For all $a,b\in \mathcal{A}$,
	\begin{gather*}
		\Delta_{\leftarrow}(a\succ b)=0\iff (a\in\Prim_{\rm Coass}(\mathcal{A}) \  \text{\rm and} \  b\in \Prim_{\leftarrow}(\mathcal{A})), \\
		\Delta_{\rightarrow}(a\prec b)=0\iff (a\in\Prim_{\rm Coass}(\mathcal{A}) \  \text{\rm and} \  b\in \Prim_{\rightarrow}(\mathcal{A})).
	\end{gather*}
\end{Prop}
 The proposition above gives two maps
\begin{align*}
	&\theta_{\succ}\colon \ \begin{cases}
		\Prim_{\rm Coass}(\mathcal{A})\otimes \Prim_{\leftarrow}(\mathcal{A}) \rightarrow \Prim_{\leftarrow}(\mathcal{A}),\\
		a\otimes b \mapsto a\succ b,
	\end{cases} \\
&\theta_{\prec}\colon \ \begin{cases}
	\Prim_{\rm Coass}(\mathcal{A})\otimes \Prim_{\rightarrow}(\mathcal{A}) \rightarrow \Prim_{\rightarrow}(\mathcal{A}), \\
	a\otimes b \mapsto a\prec b.
\end{cases}
\end{align*}

Moreover, the space $\Prim_{\rm Codend}(\mathcal{A})$ is an algebra with $\cdot$.

\begin{defi}
	A dipterous algebra $(H,*)$ is an associative algebra with an additional bilinear operation denoted by $\prec$ such that, for all $a,b,c\in H, (a\prec b)\prec c= a\prec (b*c)$.
\end{defi} We notice that $\mathcal{A}$ is also a \emph{dipterous} algebra, as a consequence $\mathcal{A}$ is cofree referring to~\cite{diptere}. In other words, there exists a vector space $V$ such that $\coT(V)\simeq \mathcal{A}$ as coalgebras. For more details see~\cite{diptere}.

\paragraph{Series of coassociative primitives:}

using the results of the previous paragraph and the results $59$ (i.e., $\Prim(\coT(V))=V$) and $46$ (i.e, $F_{T(V)}(X)=\frac{1}{1-F_V(X)}$ where $F_B$ is the formal series of the graded space $B$ such that $B_0=(0)$) from~\cite{algHopf}, we obtain
\begin{align*}
	F_{T(V)}(X)=F_\mathcal{A}(X)=\frac{1}{1-F_{\Prim_{\rm Coass}(\mathcal{A})}(X)},
\end{align*}
which implies $F_{\Prim(\mathcal{A})}(X)=\frac{F_\mathcal{A}(X)-1}{F_\mathcal{A}}.$
But, $F_\mathcal{A}(X)=1+R(X).$ As a consequence, we find
\[
F_{\Prim_{\rm Coass}(\mathcal{A})}(X)=\frac{R(X)}{1+R(X)}.
\]

\begin{Prop}\label{nbcoass}
	We have
	\[
	F_{\Prim_{\rm Coass}(\mathcal{A})}(X)=\frac{P(X)}{X}-1.
	\]
	In other words,
	\[
	F_{\Prim_{\rm Coass}(\mathcal{A})}(X)=\sum_{n=1}^{+\infty} \mathcal{A}_{n+1}X^n,
	\]
	where $\mathcal{A}_0=0$, $\mathcal{A}_1=1$, $\mathcal{A}_2=1$ and for all $n\geq 3$, $\mathcal{A}_n=2a_{n-2}.$
\end{Prop}
\begin{proof}
		From~\eqref{cal:PrimCoass}, we get
\begin{align*}
	F_{\Prim_{\rm Coass}(\mathcal{A})}(X)=\frac{R(X)}{1+R(X)}
						 =X+2XR(X).
\end{align*}
And we also have, according to Remark~\ref{Rq:series},
\begin{align*}
	\frac{P(X)}{X}-1&=\frac{X+X^2+2X^2R(X)}{X}-1=X+2XR(X). \tag*{\qed}
\end{align*}
\renewcommand{\qed}{}
\end{proof}

Finally, we obtain
\begin{thm}\label{thm:coassincodend}
	For all $n\in\N$, we define
	\begin{equation*}
		\theta_n\colon \ \begin{cases}
			\Prim_{\rm Coass}(\mathcal{A})_n\otimes \big\langle \Y \big\rangle \rightarrow \Prim_{\rm Codend}(\mathcal{A})_{n+1}, \\
			a\otimes \Y \mapsto a\cdot \Y.
		\end{cases}
	\end{equation*}
Then for all $n\in\N$, $\theta_n$ is an isomorphism of vector spaces.
\end{thm}
\begin{proof}
	Let $n\in\N.$
	By Proposition $\ref{nbcoass}$, we know that $\dim\left(\Prim_{\rm Coass}(\mathcal{A})_{n}\right)=A_{n+1}$. From Proposition $\ref{nbcodend}$, we obtain $\dim\left(\Prim_{\rm Codend}(\mathcal{A})_{n+1}\right)=A_{n+1}$.
	But
	\[
	\dim\big(\Prim_{\rm Coass}(\mathcal{A})_{n}\otimes \big\langle \Y \big\rangle\big)=\dim(\Prim_{\rm Codend}(\mathcal{A})_{n+1}).
	\]
	Consider $(a_i)_{i\in \IEM{1}{A_{n+1}}}$ a basis of $\Prim_{\rm Coass}(\mathcal{A})_n$.
	According to Theorem $\ref{thm:inthom}$, the following family $\big(a_i\cdot \Y\big)_{i\in\IEM{1}{A_{n+1}}}$ is free in $\Prim_{\rm Codend}(\mathcal{A})_{n+1}$.
	
	For reasons of dimensions, we conclude that $\theta_n$ is an isomorphism.
\end{proof}

\begin{Rq}\label{Rq:Primcoassgen}
	We fix $(\mathcal{A},\prec,\succeq,\Delta_{\rightarrow},\Delta_{\leftarrow})$ a $(2,2)$-dendriform bialgebra structure.
	 For all vector spaces $V$ over $\K$, we denote by $\Dend(V)$ the free dendriform algebra generated by $V$.
	Referring to~\cite[Theorem~4.6]{euleridem}, the brace algebras $\mathfrak{B}(V)$ and $\Prim_{\rm Coass}(\Dend(V))$ are isomorphic for any vector space $V$. For more information about brace algebra, we refer to~\cite{Chapbrace,Brace}. Thanks to~\cite[Theorem~35 and Corollary~36]{bidend}, we have $\mathcal{A}=\Tridend\big(\Y\big)=\Dend(\Prim_{\rm Codend}(\mathcal{A}))$. As a~consequence, we can compute recursively $\Prim_{\rm Coass}(\mathcal{A})$ and $\Prim_{\rm Codend}(\mathcal{A})$ the following way:
	\begin{enumerate}\itemsep=0pt
		\item Suppose we know $\Prim_{\rm Coass}(\mathcal{A})_n.$
		\item Compute $\Prim_{\rm Codend}(\mathcal{A})_{n+1}$ with Theorem~\ref{thm:coassincodend}.
		\item Use $\mathfrak{B}(\Prim_{\rm Codend}(\mathcal{A}))$ in order to compute $\Prim_{\rm Coass}(\mathcal{A})_{n+1}$ from $\Prim_{\rm Coass}(\mathcal{A})_k$, $k\leq n$ and $\Prim_{\rm Codend}(\mathcal{A})_{n+1}$.
	\end{enumerate}
\end{Rq}

\subsection[Quotient of A by langle A\^{}+ cdot A\^{}+ rangle]{Quotient of $\boldsymbol{\mathcal{A}}$ by $\boldsymbol{\big\langle \mathcal{A}^+\cdot \mathcal{A}^+ \big\rangle}$}

Consider $\mathcal{A}$ the free tridendriform algebra with one generator. We remind that we write $\mathcal{A}=1_\mathcal{A}\K\oplus \mathcal{A}^+$.
We define $\big\langle \mathcal{A}^+\cdot \mathcal{A}^+ \big\rangle$ as the tridendriform ideal generated by $\mathcal{A}^+\cdot \mathcal{A}^+.$
\begin{Lemme}
	 The tridendriform ideal $\big\langle \mathcal{A}^+\cdot \mathcal{A}^+ \big\rangle$ is a $(3,2)$-dendriform biideal of $(\mathcal{A},\prec,\succ,\allowbreak \Delta_{\leftarrow},\Delta_{\rightarrow})$.
\end{Lemme}
\begin{proof}
	By definition, it is a tridendriform ideal.
		
We check $I\coloneqq\big\langle \mathcal{A}^+\cdot \mathcal{A}^+ \big\rangle$ is a dendriform coideal.
		First, as $\varepsilon(|)=1$ and is zero on $\mathcal{A}^+$, it is easily seen that $\varepsilon(I)=(0)$. Then for all $s_1,s_2\in \mathcal{A}^+$,
		\begin{gather*}
			\Delta_{\leftarrow}(s_1\cdot s_2)
			=s_1's'_{2 \leftarrow}\otimes s_1''\cdot s''_{2\leftarrow}+s'_{2 \leftarrow}\otimes s_1\cdot s''_{2 \leftarrow} \in \mathcal{A}\otimes I, \\
			\Delta_{\rightarrow}(s_1\cdot s_2)
=s_1's'_{2 \rightarrow}\otimes s_1''\cdot s''_{2 \rightarrow}+s_1'\otimes s_1''\cdot s_2+s'_{2 \rightarrow}\otimes s_1\cdot s''_{2 \rightarrow}\in \mathcal{A}\otimes I.
		\end{gather*}
		As a consequence, $\Delta(s)\in \mathcal{A}\otimes I + I\otimes \mathcal{A}$ for $s=s_1\cdot s_2$. Then let $x$ be an element of $I$ such that $\Delta(x)\in \mathcal{A}\otimes I+I\otimes \mathcal{A}$. Let $y\in \mathcal{A}$ and $(\ltimes,\rtimes)\in\{\prec,\cdot,\succ,*\}^2$. Then denoting $\Delta(x)=|\otimes x+x\otimes |+\sum x'\otimes x''$ and $\Delta(y)=|\otimes y+y\otimes |+\sum y'\otimes y''$, we have
		\begin{gather*}
			\Delta(x)\ltimes \Delta(y) =|\otimes x\ltimes y+y\otimes x\ltimes |+x\otimes |\ltimes y+x\ltimes y\otimes |\\
			 \hphantom{\Delta(x)\ltimes \Delta(y) =}{}+\sum x'\otimes x''\ltimes y+\sum x'*y\otimes x''\ltimes |+\sum x'*y'\otimes x''\ltimes y'', \\
			\Delta(y)\rtimes\Delta(x) =|\otimes y\rtimes x+x\otimes y\rtimes |+y\otimes |\rtimes x+y\rtimes x\otimes |\\
			 \hphantom{\Delta(y)\rtimes\Delta(x) =}{}+\sum y'\otimes y''\rtimes x+\sum y'*x\otimes y''\rtimes |+\sum y'*x'\otimes y''\rtimes x''.
		\end{gather*}
		In both expressions, each term is either $0$ or an element of $\mathcal{A}\otimes I+I\otimes \mathcal{A}$ using the fact that $I$ is a~tridendriform ideal. We have proved that $\Delta(x\ltimes y)$ and $\Delta(y\rtimes x)$ belong to $\mathcal{A}\otimes I+I\otimes \mathcal{A}$. Using similar ideas, we easily prove the same results for $\Delta_{\leftarrow}(x\ltimes y)=\Delta(x)\ltimes \Delta_{\leftarrow}(y)$, $\Delta_{\leftarrow}(y\rtimes x)=\Delta(y)\rtimes \Delta_{\leftarrow}(x)$, $\Delta_{\rightarrow}(x\ltimes y)=\Delta(x)\ltimes \Delta_{\rightarrow}(y)$ and $\Delta_{\rightarrow}(y\rtimes x)=\Delta(y)\rtimes \Delta_{\rightarrow}(x)$. So it proves that $\big\langle \mathcal{A}^+ \cdot \mathcal{A}^+\big\rangle$ is a dendriform coideal.

So the proof is complete.
\end{proof}

As a consequence, we can divide $\mathcal{A}$ by the $(3,2)$-dendriform biideal $\big\langle \mathcal{A}^+\cdot \mathcal{A}^+ \big\rangle$.
Referring to the Remark~\ref{algo}, we will give a characterization to describe the trees in $\big\langle \mathcal{A}^+\cdot \mathcal{A}^+ \big\rangle$.

\begin{Lemme}
	Let $t$ be a tree of $\mathcal{A}$. Then $t$ is not a binary tree \ssi{} $t\in \big\langle \mathcal{A}^+\cdot \mathcal{A}^+ \big\rangle$.
\end{Lemme}
\begin{proof}
	
	Indeed, suppose that $t=t_1\cdot t_2$ where $t_1=t_1^{(0)}\vee \dots\vee t_1^{(k_1)}$ and $t_2=t_2^{(0)}\vee \dots\vee t_2^{(k_2)}$ with $k_1\geq 1$ and $k_2\geq 1$ because $t_1,t_2\in \mathcal{A}^+$.
	As a consequence,
	\begin{align*}
		&t=\big(t_1^{(0)}\vee \dots\vee t_1^{(k_1)}\big)\cdot \big(t_2^{(0)}\vee \dots\vee t_2^{(k_2)}\big)
		 =t_1^{(0)}\vee \dots\vee \big(t_1^{(k_1)}*t_2^{(0)}\big)\vee \dots\vee t_2^{(k_2)}.
	\end{align*}
	So, this is clear that $t$ is not a binary tree.
	
	Conversely, consider a non-binary tree $t$. We want to show that $t$ has at least one writing with $\cdot$. To prove this, we proceed by induction on the number of leaves of $t$. We denote $n$ the number of leaves of $t$.
	To initialize the induction, we need to begin with $n=3$. The only tree which is not binary is $\balais=\Y\cdot \Y$. 	
	To prove the heredity, we consider that $n>3$ and we suppose that for all $3\leq m<n$, non-binary trees with $m$ leaves are elements of $\big\langle \mathcal{A}^+\cdot \mathcal{A}^+ \big\rangle$.
	Let~$t$ be a tree with $n$ leaves. We have the following writing:
	\[
	t=t^{(0)}\vee \dots\vee t^{(k)}.
	\]
	
\textit{Case 1}: $k\geq 2$.
		We choose $i\in\IEM{1}{k-1}$.
	Then writing $t_1=t^{(0)}\vee\dots\vee t^{(i)}\vee |$ and $t_2=t^{(i+1)}\vee \dots\vee t^{(k)}$, we get
	\begin{gather*}
	t=t_1\cdot t_2.
	\end{gather*}
	
\textit{Case 2}: $k=1$. As a consequence, $t=t^{(0)}\vee t^{(1)}$ where each of these trees owns at least one leaf. But, as~$t$ is not binary either $t^{(0)}$ is not binary, either $t^{(1)}$ is not binary. In the two cases, each of these trees has strictly less leaves than $t$. By the induction hypothesis, we can suppose that $t^{(0)}\in \big\langle \mathcal{A}^+\cdot \mathcal{A}^+ \big\rangle$ without loss of generality.
	But, $\big\langle \mathcal{A}^+\cdot \mathcal{A}^+ \big\rangle$ is a tridendriform ideal. This implies that $t\in \big\langle \mathcal{A}^+\cdot \mathcal{A}^+ \big\rangle$ because $t=\left(t^{(0)}\succ \Y\right)\prec t^{(1)}$.
\end{proof}

This allows us to conclude that \smash{$\faktor{\mathcal{A}}{\langle \mathcal{A}^+\cdot \mathcal{A}^+ \rangle}$} and $ \bigoplus_{n\geq 0} \K\PBT_n$ are isomorphic as vector~spaces.
Finally, the quotient space of $\mathcal{A}$ by $\big\langle \mathcal{A}^+\cdot \mathcal{A}^+ \big\rangle$ is a $(3,2)$-dendriform bialgebra with its natural structure of quotient space. But, the product $\cdot$ is $0$ on this space. In particular, the bialgebra we just obtained can be seen as a $(2,2)$-dendriform bialgebra. More precisely, in \smash{$\faktor{\mathcal{A}}{\langle \mathcal{A}^+\cdot \mathcal{A}^+ \rangle}$}
\[
	*=\prec+\succ.
\]
This product corresponds to the one written in~\cite[Proposition 3.2]{LR}. Moreover, it is not hard to see that the coproduct of the quotient space is the same as the one in~\cite[Proposition~3.3]{Eclair}. This algebra is the \emph{Loday--Ronco algebra}.
Referring to~\cite{LR}, the Proposition $3.2$ says that for $t$,~$t'$ two binary trees written $t=t_1\vee t_2$ and $t'=t_1'\vee t_2'$, we have
\[
t*t'=t_1\vee (t_1'*t_2)+(t*t_1')\vee t_2'.
\]
This is exactly the descriptions of the products $\prec$ and $\succ$ restricted to binary trees.
\begin{Prop}
The coproduct given in~{\rm \cite{LR}} on a tree $t=t_1\vee t_2$ with $t_1\in T_n$ and $t_2\in T_m$ is defined by
	\begin{equation*}
		\Delta(t)=\sum_{j,k}\big( t_1^{(1)(j)}* t_2^{(1)(k)} \big)\otimes \big( t_1^{(2)(n-j)}\vee t_2^{(2)(m-k)} \big)+t\otimes |,
	\end{equation*}
	where $ \Delta(t_1)=\sum_{j=0}^n t_1^{(1)(j)}\otimes t_1^{(2)(n-j)}$ and $ \Delta(t_2)=\sum_{k=0}^m t_2^{(1)(k)}\otimes t_2^{(2)(m-k)}$.
	This coproduct is the same as the one on \smash{$\faktor{\mathcal{A}}{\langle \mathcal{A}^+\cdot \mathcal{A}^+ \rangle}$}.
\end{Prop}
\begin{proof}
We will proceed by induction over $n$ the number of leaves of $t$ to show the coproduct~$\Delta$ of Loday--Ronco algebra is the same as the one over \smash{$\faktor{\mathcal{A}}{\langle \mathcal{A}^+\cdot \mathcal{A}^+ \rangle}$} denoted by $\overline{\Delta}$.

\textit{Initialization}:
if $n=1$, i.e., $t=|$, then $\Delta(|)=\overline{\Delta}(|)=|\otimes |$. If $n=2$, i.e., $t=\Y$, we get
\begin{gather*}
	\overline{\Delta} \big(\Y\big)=\Y\otimes | +|\otimes \Y, \qquad \Delta\big(\Y\big)= |*|\otimes |\vee |+\Y\otimes |=\Y\otimes | + |\otimes \Y.
\end{gather*}

\textit{Heredity}: let $n\in\N$ such that $n\geq 3$, suppose that for all $n'<n$ we have the equality between~$\Delta$ and $\overline{\Delta}$ for all trees with $n'$ leaves. We check that it is also true for trees with $n$ leaves. Let us write $t=t_1\vee t_2$ such that $t_1$ has $n_1$ leaves and $t_2$ has $n_2$ leaves. Two mutually excluding cases can happen:
\begin{itemize}\itemsep=0pt
	\item The case $n_2+1\neq n$. Using the properties of $\overline{\Delta}$, we get
	\begin{align*}
		& \overline{\Delta}(t)=\overline{\Delta}(t_1\vee t_2)=\overline{\Delta}(t_1\succ (|\vee t_2))
		 =\overline{\Delta}(t_1)\succ \overline{\Delta}(|\vee t_2).
	\end{align*}
Applying the induction hypothesis once, we get
	\begin{align*}
		\overline{\Delta}(t)
		={}&\Delta(t_1)\succ \overline{\Delta}(|\vee t_2) \\
		={}&\Bigg( \sum_{j=0}^{n_1}t_1^{(1)(j)}\otimes t_1^{(2)(n_1-j)} \Bigg)\succ \bigg( |\vee t_2\otimes | + \!\!\sum_{c \text{ admissible cut of }t_2 }\!\!\!\!G^c(t_2)\otimes |\vee P^c(t_2)\bigg) \\
		={}&\Bigg( \sum_{j=0}^{n_1}t_1^{(1)(j)}\otimes t_1^{(2)(n_1-j)} \Bigg)\succ \Bigg(|\vee t_2\otimes |+\sum_{k=0}^{n_2}t_2^{(1)(k)}\otimes |\vee t_2^{(2)(n_2-k)}\Bigg) \\
		={}&\sum_{j,k} t_1^{(1)(j)}*t_2^{(1)(k)}\otimes t_1^{(2)(n_1-j)}\vee t_2^{(2)(n_2-k)}+t\otimes |.
	\end{align*}
So, in this case we have equality between $\overline{\Delta}$ and $\Delta$.
	\item The case $n_1+1\neq n$ is done the same way. That is why we omit some details,
	\begin{gather*}
		\overline{\Delta(t)}
		 =\overline{\Delta}((t_1\vee |)\prec t_2) \\
		 \hphantom{\overline{\Delta(t)}}{}=\sum_{j,k} t_1^{(1)(j)}* t_2^{(1)(k)}\otimes \big( \big(t_1^{(2)(n_1-j)}\vee |\big)\prec t_2^{(2)(n_2-k)}\big) \\
		 \hphantom{\overline{\Delta(t)=}}{}+\sum_k (t_1\vee |)*t_2^{(1)(k)}\otimes |\prec t_2^{(2)(n-k)} \\
		 \hphantom{\overline{\Delta(t)}}{}=\sum_{j,k} t_1^{(1)(j)}* t_2^{(2)(k)}\otimes t_1^{(2)(n_1-j)}\vee t_2^{(n_2-k)}+t\otimes |.
	\end{gather*}
\end{itemize}
Therefore, we have the result.
\end{proof}

\subsection{Conclusion}

Thanks to $\cite{Trial}$, we have given a bialgebra structure to the free tridendriform algebra $\mathcal{A}$ such that it is a $(3,2)$-dendriform bialgebra structure. Considering the graded dual of $\mathcal{A}$, we have defined $(2,3)$-dendriform bialgebras.
Using $\cite{Eclair}$, we have identified the graded dual of $\mathcal{A}$ with TSym as bialgebras. Moreover, it shows that $\TSym$ can be given a $(2,3)$-dendriform bialgebra structure. It shows the interactions between the combinatorics on trees and the dendriform and tridendriform algebras.

Some ideas are not explored in this paper. For example, can we get other quotient spaces from the free $(3,2)$-dendriform bialgebra? Can we find a description of the coassociative primitives of this algebra obtained with Remark~\ref{Rq:Primcoassgen}? Or can we generalize our ideas to other $(n,m)$-dendriform structures and maybe give a more general definition?

\subsection*{Acknowledgements}

I want to thank all members of the LMPA at Universit\'e du Littoral C\^ote d'Opale for welcoming me to prepare my PhD. I especially thank Lo\"{\i}c Foissy, my PhD advisor.
I thank the other PhD students of the LMPA for useful discussions which give me some ideas. I also want to thank the referees for the careful reading of this paper.
The author acknowledges support from the grant ANR-20-CE40-0007
\emph{Combinatoire Alg\'ebrique, R\'esurgence, Probabilit\'es Libres et Op\'erades}.

\pdfbookmark[1]{References}{ref}
\LastPageEnding

\end{document}